\documentclass[10pt]{amsart}
\usepackage{ amssymb,latexsym, amscd,pb-diagram}
\usepackage{sseq}
\usepackage[all]{xy}
\usepackage[square, numbers]{natbib}
\usepackage{graphicx}
\usepackage{mathrsfs}
\usepackage{color}
\usepackage{amsmath}
\usepackage{tikz}
\usepackage{dsfont}
\usetikzlibrary{matrix}

 \newtheorem{theorem}{Theorem}[section]
 \newtheorem{cor}[theorem]{Corollary}
  
 \newtheorem{lemma}[theorem]{Lemma}
 \newtheorem{proposition}[theorem]{Proposition}
 \theoremstyle{definition}
 
 \theoremstyle{definition}

 \newtheorem{rem}[theorem]{Remark}
 \numberwithin{equation}{section}

\newcommand{\ben}{\begin{equation}}
\newcommand{\een}{\end{equation}}

\newcommand{\integer}{\ensuremath{{\mathbb Z}}}

\newcommand{\complex}{\ensuremath{{\mathbb C}}}

\newcommand{\Cx}{\ensuremath{{\mathbb C}^*}}

\newcommand{\DD}{{\mathcal D}}

\newcommand{\ZZ}{{\mathcal Z}}

\newcommand{\VV}{{\mathcal V}}

\newcommand{\CC}{\mathcal{C}}

\newcommand{\MM}{\mathcal{M}}

\newcommand{\Hom}{\mathrm{Hom}}

\newcommand{\Maps}{\mathrm{Maps}}

\newcommand{\To}{\longrightarrow}

\newcommand{\IF}{{\mathbb{F}}}

\newcommand{\trl}{{\vartriangleleft}}
\newcommand{\trr}{{\vartriangleright}}

\newcommand{\Map}{\ensuremath{{\mathrm{Map}}}}

\newcommand{\tensor}{\ensuremath{\otimes}}

\newcommand{\Tot}{\mbox{\rm Tot\,}}

\newcommand{\Obj}{\mbox{\rm Obj\,}}

\newcommand{\Fun}{\operatorname{Fun}}

\newcounter{commentcounter}

\begin{document}

\title[On the Classification of Pointed Fusion Categories]{On the Classification of Pointed Fusion Categories up to weak Morita Equivalence}

\thanks{The author acknowledges the financial support of the Max Planck Institute of Mathematics in Bonn, Germany, and 
of COLCIENCIAS through grant number FP44842-617-2014 of the Fondo Nacional de Financiamiento para la Ciencia, la Tecnolog\'ia
 y la Inovaci\'on.}
\author{Bernardo Uribe}
\address{Departamento de Matem\'{a}ticas y Estad\'istica, Universidad del Norte, Km.5 V\'ia Antigua a Puerto Colombia, 
Barranquilla, Colombia.}
\email{bjongbloed@uninorte.edu.co}
\begin{abstract}
A pointed fusion category is a rigid tensor category with finitely many isomorphism classes of simple objects which moreover
are invertible.  Two tensor categories $\CC$ and $\DD$ are {weakly Morita equivalent} if there exists an
  indecomposable right module category $\MM$ over $\CC$ such that $\Fun_\CC(\MM,\MM)$ and $\DD$
   are tensor equivalent. We use the Lyndon-Hochschild-Serre spectral sequence associated
   to abelian group extensions to give necessary and sufficient conditions in terms of cohomology classes for two pointed fusion categories
   to be weakly Morita equivalent. This result may permit to classify the equivalence classes of pointed fusion categories of any given global dimension.

\end{abstract}

\maketitle

\section*{Introduction}

Pointed fusion categories are rigid tensor categories with finitely many isomorphism classes of simple objects with the property
that all simple objects are invertible. Any pointed fusion category $\CC$ is equivalent to the fusion category $Vect(G,\omega)$ of
complex vector spaces graded by the finite group $G$ together with the associativity constraint defined by the 3-cocycle
$\omega \in Z^3(G,\Cx)$. Whenever we have a right module category $\MM$ over $\CC$ we can define
the dual category $\CC_\MM^*:= \Fun_\CC(\MM,\MM)$ which becomes a tensor category via composition of functors. Whenever
$\CC$ is a fusion category and $\MM$ is an indecomposable fusion category, the dual category $\CC_\MM^*$ is also a fusion category \cite[\S 2.2]{Ost-2}. An indecomposable module category $\MM$ of $Vect(G,\omega)$ may be defined by $\MM=\MM(K, \mu)$
 where $K$ is the space of cosets $K :=A \backslash G$ for $A$ a subgroup of $G$ and 
   $\mu \in C^2(G, \Map(K, \Cx))$ is a cochain that satisfies the equation
  $   \delta_G \mu^{-1} = \omega$.
  Two tensor categories $\CC$ and $\DD$ are {{weakly Morita equivalent}} if there exists an
  indecomposable right module category $\MM$ over $\CC$ such that $\CC_\MM^*$ and $\DD$
   are tensor equivalent \cite[Def 4.2]{MugerI}. 
   
   Now, if we have two pointed fusion categories $Vect(G,\omega)$ and $Vect(\widehat{G},\widehat{\omega})$, what are the necessary and
   sufficient conditions for them to be weakly Morita equivalent? This question was raised in \cite{Davydov, Movshev}, it was
   answered by Davydov \cite[Cor. 6.2]{Davydov} for the case on which both $\omega$ and $\widehat{\omega}$ were trivial, and the general case was answered by Naidu in \cite[Theorem 5.8]{Naidu}  in terms of the properties that $A$, $\omega$ and $\mu$ need to satisfy. Nevertheless these conditions were given 
   in equations that a priori had no interpretation in terms of known cohomology classes.    

We continue the work started by Naidu in \cite{Naidu} and frame all the calculations done there in the language of the double 
complex associated to an abelian group extension which induces the Lyndon-Hochschild-Serre spectral sequence. By doing so we are able
to obtain in Corollary \ref{omega in 2,1 and 3,0}
cohomological conditions on $\omega$ in order for the tensor category $Vect(G,\omega)_{\MM(A \backslash G, \mu)}^*$ to be
 pointed, namely that $\omega$ must be cohomologous to a cocycle appearing
 in $C^{2,1}\oplus C^{3,0}$ of the double complex which induces
 the Lyndon-Hochschild-Serre spectral sequence associated to the extension
    $1 \to A \to G \to K \to 1$.

With the previous result at hand, we construct explicit representatives of $\omega$ and $\mu$ in terms of coordinates and 
we determine explicitly the groups $\widehat{G}$ and the cocycles $\widehat{\omega}$. The main result of this paper is Theorem
\ref{main theorem} in which we give the necessary and sufficient conditions for the categories
$Vect(H,\eta)$ and $Vect(\widehat{H},\widehat{\eta})$ to be weakly Morita equivalent. We may summarize
the conditions as follows: $Vect(H,\eta)$ and $Vect(\widehat{H},\widehat{\eta})$ are weakly Morita equivalent if and only if
there exist isomorphisms of groups 
 $\phi : A \rtimes_F K \stackrel{\cong}{\to} H$ and 
$ \widehat{\phi} :  K \ltimes_{\widehat{F}} {{\mathbb{A}}} \stackrel{\cong}{\to} \widehat{H}$
 for some finite group $K$ acting on the abelian group $A$,
 with $F \in Z^2(K, A)$ and $\widehat{F} \in Z^2(K, {{\mathbb{A}}})$ where ${{\mathbb{A}}} := \Hom(A, \Cx)$, such that
 both $[\widehat{F}]$ and $[F]$ survive respectively the LHS spectral sequence for the groups $A \rtimes_F K$ and $K \ltimes_{\widehat{F}} {{\mathbb{A}}}$, and such that $\phi^* \eta$ is cohomologous to
 $$  \omega((a_1,k_1),(a_2,k_2),(a_3,k_3)) := \widehat{F}(k_1,k_2)(a_3) \ \epsilon(k_1,k_2,k_3)$$
  and 
 $\widehat{\phi}^*\widehat{\eta}$ is cohomologous to 
$$  \widehat{\omega}((k_1, \rho_1),  (k_2,\rho_2),(k_3 ,\rho_3)) := \epsilon(k_1,k_2,k_3) \ \rho_1(F(k_2,k_3))$$
where $\epsilon: K^3 \to \Cx$ satisfies $\delta_K \epsilon = \widehat{F} \wedge F$. 

Theorem \ref{main theorem} may be used to determine the weak Morita equivalence classes of pointed fusion categories of a given
global dimension but the cohomological calculations can become very elaborate and are beyond the scope of this article.
Nevertheless in
section \S \ref{section examples} we include a calculation on which we show how Theorem
\ref{main theorem} can be used to show that there are only seven weak Morita equivalence classes of pointed fusion categories of
global dimension four, and in order to calculate the pointed fusion categories which are weakly Morita equivalent to $Vect(Q_8,\eta)$
for the quaternion group $Q_8$.

\section{Preliminaries} \label{section Preliminaries}
\subsection{Abelian group extensions} \label{subsection Abelian group extensions}
Consider the short exact sequence of finite groups
\begin{eqnarray} \label{extension of G by A and K}
1 \To A \To G \To K \To 1
\end{eqnarray}
with $A$ abelian. Consider $u : K \to G$ any section
of the projection map $p: G \to K, p(g)=(Ag)$ such that $u(1_K)=1_G$ and denote the right $G$-action on $K$ 
by $$k \trl g:= p((u(k)g)$$ for $k \in $ and $g \in G$. The elements $u(k)g$ and $u(k \trl g)$ differ by an element $\kappa_{k,g} \in A$ satisfying the equation
\begin{align}
u(k)g = \kappa_{k,g} u(k \trl g) \label{equation of kappa}
\end{align}
which furthermore satisfies the relation
$$\kappa_{k,g_1g_2}= \kappa_{k,g_1} \kappa_{k \trl g_1,g_2}$$
for $ k \in K$ and $g_1,g_2 \in G$.
  Since $A$ is an abelian normal subgroup $G$, there is an induced $K$-left action on $A$ by conjugation: $${}^ka: = u(k)a u(k)^{-1}$$ for $k \in K$ and $a \in A$.

Since the isomorphism class of the extension \eqref{extension of G by A and K} can be classified by the cohomology
class of the cocycle $F \in Z^2(K, A)$, i.e. a map $F : K \times K \to A$ such that
$$\delta_K F (k_1,k_2,k_3)= {}^{k_1}F(k_2,k_3)F(k_1k_2,k_3)^{-1}F(k_1,k_2k_3) F(k_1,k_2)^{-1}=1,$$
without loss of generality we will further assume that
$$G := A \rtimes_F K$$
where the product structure of $G$ is given by the formula
$$(a_1,k_1) (a_2,k_2) := (a_1 ({}^{k_1}a_2) F(k_1,k_2),k_1k_2).$$

With this explicit choice of the group $G$, we choose the function $u: K \to G$ to be $u(k):=(1_A,k)$ and therefore
we have that $$\kappa_{k_1,(a,k_2)}= {}^{k_1}aF(k_1,k_2)$$ thus obtaining 
$F(k_1,k_2)= \kappa_{k_1,(1,k_2)}$. We furthermore have that for $x \in K$ and $g=(a,k) \in G$ 
$$x \trl g= x \trl (a,k)=xk.$$

Denote the dual group ${{\mathbb{A}}} := \Hom(A, \Cx)$ and note that there is an induced
$K$-right action on ${{\mathbb{A}}}$  defined as $\rho^k(a):= \rho({}^ka)$  for $\rho \in {{\mathbb{A}}}$ and $k \in K$.

\subsection{Cohomology of groups and the Lyndon-Hochschild-Serre spectral sequence}
\label{subsection Cohomology of groups and the Lyndon-Hochschild-Serre spectral sequence}

In what follows we will construct an explicit double complex whose cohomology calculates the cohomology of the group $G$, and whose associated spectral sequence recovers the Lyndon-Hochschild-Spectral (LHS) spectral sequence of the extension
\eqref{extension of G by A and K}.

Endow the set $\Map(K, \Cx)$ with the left $G$-action $(g \trr f) (k):= f(k \trl g)$ where $g \in G$, $k \in K$ and $f : K \to \Cx$, and consider the complex $C^*(G,\Map(K, \Cx))$ with elements
normalized chains
$$C^q(G,\Map(K, \Cx)):= \{ f : K \times G^q \to \Cx | f(k;g_1,...,g_q)=1 \ \ \mbox{whenever some} \ \ g_i=1 \}$$
and boundary map
\begin{align}
(\delta_G f)(k ; g_1,...,g_q) = f(k \trl g_1;g_2,...,g_q) \prod_{i=1}^{q-1}f(k;g_1 ,..,g_ig_{i+1} &,...,g_q)^{(-1)^i} \nonumber \\
& f(k; g_1,...,g_{q-1})^{(-1)^{q}}. \label{differential G}
\end{align}

Since the natural morphism of groupoids, defined by the inclusion of the group $A$ into the action groupoid defined by the right action of $G$ on $K$,  is an equivalence of categories, we have that the restriction map
$$\psi : C^*(G,\Map(K, \Cx)) \to C^*(A, \Cx), \ \ \psi(f)(a_1,...,a_q):= f(1_K;a_1,...,a_q)$$
is a morphism of complexes which induces an isomorphism in cohomology
$$\widetilde{\psi}: H^*(G,\Map(K, \Cx)) \stackrel{\cong}{\to} H^*(A, \Cx).$$
The inverse map could be constructed at the level of cocycles  as follows
\begin{lemma} \label{lemma varphi}
The map $\varphi: C^q(A, \Cx) \to C^q(G,\Map(K, \Cx))$
$$\varphi(\alpha)(k;g_1,...,g_q):= \alpha(\kappa_{k,g_1}, \kappa_{k \trl g_1,g_2},...,\kappa_{k\trl g_1g_2...g_{q-1},g_q})$$
defines a map of complexes inducing an isomorphism in cohomology $\widetilde{\varphi}: H^*(A, \Cx) \stackrel{\cong}{\to} H^*(G,\Map(K, \Cx))$
which is the inverse of the map $\widetilde{\psi}$.
\end{lemma}
\begin{proof}
On the one hand we have
\begin{align*}
\delta_G \varphi(\alpha)(k;g_1,...,g_p) = &   \varphi(\alpha)(k \trl g_1;g_2,...,g_q) \prod_{i=1}^{q-1} \varphi(\alpha)(k;g_1 ,..,g_ig_{i+1} ,...,g_q)^{(-1)^i}\\
&  \varphi(\alpha)(k; g_1,...,g_{q-1})^{(-1)^{q}}\\
=& \alpha(\kappa_{k \trl g_1,g_2}, \kappa_{k \trl g_1g_2,g_3},..., \kappa_{k \trl g_1...g_{q-1}, g_q})\\
&  \prod_{i=1}^{q-1} \alpha(\kappa_{k,g_1},\kappa_{k \trl g_1,g_2} ,...,\kappa_{k \trl g_1...g_{i-1},g_ig_{i+1}},..., \kappa_{k \trl g_1...g_{q-1},g_q})^{(-1)^i}\\
 & \alpha((\kappa_{k,g_1},\kappa_{k \trl g_1,g_2} ,...,\kappa_{k \trl g_1...g_{q-2},g_{q-1}})^{(-1)^{q}}
\end{align*}
and on the other
\begin{align*}
\varphi( \delta_G  & \alpha)(k;g_1,...,g_p) =   \delta_G  \alpha (\kappa_{k,g_1}, \kappa_{k \trl g_1,g_2},...,\kappa_{k\trl g_1g_2,...,g_{q-1},g_q})\\
=& \alpha(\kappa_{k \trl g_1,g_2}, \kappa_{k \trl g_1g_2,g_3},..., \kappa_{k \trl g_1...g_{q-1}, g_q})\\
&  \prod_{i=1}^{q-1} \alpha(\kappa_{k,g_1},\kappa_{k \trl g_1,g_2} ,...,
\kappa_{k \trl g_1...g_{i-1},g_i}\kappa_{k \trl g_1...g_{i-1}g_i,g_{i+1}},..., \kappa_{k \trl g_1...g_{q-1},g_q})^{(-1)^i}\\
 & \alpha((\kappa_{k,g_1},\kappa_{k \trl g_1,g_2} ,...,\kappa_{k \trl g_1...g_{q-2},g_{q-1}})^{(-1)^{q}}.
\end{align*}
The equality $\delta_G \varphi(\alpha)=\varphi( \delta_G   \alpha)$ follows from the identity
$$\kappa_{k \trl g_1...g_{i-1},g_ig_{i+1}}=\kappa_{k \trl g_1...g_{i-1},g_i}\kappa_{k \trl g_1...g_{i-1}g_i,g_{i+1}}.$$
Finally, the composition $\psi(\varphi(\alpha))=\alpha$ follows from the equation $\kappa_{1,a}=a$ for $a \in A$.
\end{proof}

The complex $C^*(A, \Cx)$ can be endowed with the structure of a right K-module by setting for $\alpha \in
C^q(A, \Cx)$ and $k \in K$ 
$$\alpha^k (a_1,...,a_q) := \alpha(u(k)a_1u(k)^{-1},...,u(k)a_qu(k)^{-1}),$$
and the complex  $C^*(G,\Map(K, \Cx))$ can also be endowed with the structure of a right $K$-module by setting
for $f \in C^q(G,\Map(K, \Cx))$ and $k \in K$
$$(f \trl k)(x; g_1,...,g_q):= f(kx;g_1,...,g_q).$$

The map $\varphi$ fails to be a $K$-module map; nevertheless it induces a $K$-module map at the level of cohomology

\begin{lemma} \label{lemma iso varphi}
The isomorphism $\widetilde{\varphi}: H^*(A, \Cx) \stackrel{\cong}{\to} H^*(G,\Map(K, \Cx))$ is an isomorphism of
$K$-modules.
\end{lemma}
\begin{proof}
Take $\alpha \in Z^q(A, \Cx)$ and $k \in K$. We claim that $\psi(\varphi(\alpha) \trl k)= \alpha^k$, and since
$\psi(\varphi(\alpha^k))=\alpha^k$, we conclude that $\varphi(\alpha) \trl k$ and $\varphi(\alpha^k)$ are cohomologous.
Now, let us calculate
\begin{align*}
\psi(\varphi(\alpha) \trl k)(a_1,...,a_q) =& (\varphi(\alpha) \trl k)(1;a_1,...,a_q)\\
=& \varphi(\alpha) (k;a_1,...,a_q)\\
=& \alpha(\kappa_{k,a_1}, \kappa_{k \trl a_1,a_2},...,\kappa_{k\trl a_1a_2,...,a_{q-1},a_q})\\
=& \alpha(\kappa_{k,a_1}, \kappa_{k ,a_2},...,\kappa_{k ,a_q})\\
=& \alpha(u(k)a_1u(k)^{-1}, u(k)a_2u(k)^{-1},...,u(k)a_qu(k)^{-1})\\
=& \alpha^k(a_1,a_2,...,a_q);
\end{align*}
the lemma follows.
\end{proof}

\subsubsection{Double complex} \label{subsubsection Double complex} Since $C^*(G,\Map(K, \Cx))$ is a complex of right $K$-modules, we can consider the complexes
$$C^*(K, C^q(G,\Map(K, \Cx)))$$
with $C^p(K, C^q(G,\Map(K, \Cx))$ consisting of normalized cochains
$$ \{f : K^p \to C^q(G,\Map(K, \Cx))| f(k_1,...,k_p)=1 \ \ \mbox{whenever some} \ \ k_i=1 \}$$ and whose differentials 
\begin{align*}
(\delta_K f)(k_1,...,k_p) =& f(k_2,...,k_p) \prod_{i=1}^{p-1} f(k_1,...,k_ik_{i+1},...,k_p)^{(-1)^i}(f(k_1,...,k_{p-1}) \trl k_p)^{(-1)^p}.
\end{align*}

These complexes assemble into a double complex $C^{p,q} :=C^p(K, C^q(G,\Map(K, \Cx)))$. Let us denote by
$\Tot(C^{*,*})$ the total complex associated to the double complex and let $\delta_{Tot}:=\delta_K \oplus (\delta_G)^{(-1)^p}$
be its differential.

We may filter the total complex by the degree of the $G$ cochains, thus obtaining a spectral sequence
whose first page becomes
$$E_1^{p,q}= H^p(K, C^q(G,\Map(K, \Cx))).$$
Since the $K$-modules $C^q(G,\Map(K, \Cx))$ are free $K$-modules, we conclude that the first page localizes on
the $y$-axis,
$$E_1^{0,q} = H^0(K, C^q(G,\Map(K, \Cx))) = C^q(G,\Map(K, \Cx))^K \cong C^q(G, \Cx)$$
and  $E_1^{p,q}=0$ for $p>0$. The spectral sequence collapses at the second page, with the only surviving
elements on the $y$-axis
$$E_2^{0,q}= H^q(G, \Cx).$$
Hence we have
\begin{proposition}
The inclusion of $K$-invariant cochains
$$C^*(G,\Map(K, \Cx))^K \hookrightarrow \Tot(C^*(K, C^*(G,\Map(K, \Cx))))$$
is a quasi-isomorphism. Therefore the cohomology groups
$$H^*(G,\Cx) \stackrel{\cong}{\to} H^*(\Tot(C^*(K, C^*(G,\Map(K, \Cx))))$$
are canonically isomorphic.
\end{proposition}

Filtering the double complex by the degree of the $K$ cochains we obtain the Lyndon-Hochschild-Serre
spectral sequence associated to the group extension $1 \to A \to G \to K \to 1$ (see \cite[\S 7.2]{Evens} and references therein). The first page becomes
$$E_1^{p,q} = C^p(K, H^q(G,\Map(K, \Cx)))$$
and the second page becomes
$$E_2^{p,q} = H^p(K, H^q(G,\Map(K, \Cx))).$$
Since the projection map $\widetilde{\psi}:  H^q(G,\Map(K, \Cx)) \stackrel{\cong}{\to} H^q(A,\Cx)$
is an isomorphism of $K$-modules, we conclude
\begin{proposition}[LHS spectral sequence] Filtering the total complex by the degree of the $K$-chains,
we obtain a spectral sequence whose second page is 
$$E_2^{p,q} \cong H^p(K, H^q(A,\Cx))$$
and that converges to $H^*(G,\Cx)$.
\end{proposition}

We will denote $d_i: E_i^{p,q} \to E_i^{p+i,q-i+1}$ the differentials of this spectral sequence.

\subsection{Tensor categories}
Following \cite[\S 1]{Bakalov-Kirillov}, a tensor category consist of $(\CC, \otimes, 1_\CC, \alpha, \lambda, \rho)$
where $\CC$ is a category, $\otimes : \CC \times \CC \to \CC$ is a bifunctor, $\alpha$ is the associativity constraint i.e.
a functorial isomorphism $\alpha_{UVW}: (U \otimes V) \otimes W \stackrel{\sim}{\to} U \otimes (V \otimes W)$
of functors $\CC \times \CC \times \CC \to \CC$, $1_{\CC} \in {\rm{Ob}}(\CC)$ is a unit element and $\lambda, \rho$ are functorial
isomorphisms $\lambda_V : 1_\CC \otimes V \stackrel{\sim}{\to} V$,  $\rho_V : V \otimes 1_\CC  \stackrel{\sim}{\to} V$, satisfying the pentagon axiom
$$\xymatrix{
&  ((V_1 \otimes V_2) \otimes V_3) \otimes V_4 \ar[ld]_{\alpha_{1,2,3} \otimes id_4} \ar[dr]^{\alpha_{12,3,4}} & \\
(V_1 \otimes (V_2 \otimes V_3)) \otimes V_4 \ar[d]^{\alpha_{1,23,4}} && ( V_1 \otimes V_2) \otimes (V_3 \otimes V_4) \ar[d]^{\alpha_{1,2,34}}\\
V_1 \otimes ((V_2 \otimes V_3) \otimes V_4) \ar[rr]_{id_1 \otimes \alpha_{2,3,4}} && ( V_1 \otimes (V_2 \otimes (V_3 \otimes V_4))) 
}$$
 and the triangle axiom 
$$\xymatrix{(V_1 \otimes 1_\CC) \otimes V_2 \ar[dr]_{\rho \otimes id}
\ar[rr]^\alpha && V_1 \otimes (1_\CC \otimes V_2) \ar[ld]^{id \otimes \lambda}\\
& V_1 \otimes V_2. &
}$$

\subsection{The fusion category $Vect(G, \omega)$}
A fusion category over $\complex$  is a rigid semi-simple $\complex$-linear tensor category, with only finitely many isomorphism classes of simple objects, such that the endomorphisms of the unit object is $\complex$ (see \cite{ENO}). 

For $G$ a finite group and a 3-cocycle $\omega \in Z^3(G, \Cx)$, define the category $Vect(G, \omega)$ by setting that its objects are $G$-graded complex vector spaces  $V = \bigoplus_{g \in G} V_g$, whose tensor product is
$$(V \otimes W)_g : = \bigoplus_{hk=g}V_h \otimes W_k,$$
whose associativity constraint is
$\alpha_{V_g,V_h,V_k} = \omega(g,h,k) \gamma$ with $\gamma((x \otimes y) \otimes z)= x \otimes (y \otimes z)$,
and whose left and right unit isomorphisms are $\lambda_{V_g} =\omega(1,1,g)^{-1}id_{V_g}$ and $\rho_{V_g} =\omega(g,1,1)id_{V_g}$. 
The category $Vect(G, \omega)$ is a fusion category where the simple objects are the 1-dimensional vector spaces.

 We will assume that all group cochains are normalized, and hence the left and right unit isomorphisms become identities.

 For convenience we will work with a category $\VV(G,\omega)$ which is {\it{skeletal}}, i.e. one on which isomorphic objects are equal,  and which is equivalent to $Vect(G,\omega)$.  The category $\VV(G,\omega)$ has for simple objects the elements $g$ of the group $G$, the tensor product is
 $g \otimes h =gh$ and the associativity isomorphisms are $\omega(g,h,k)id_{ghk}$.
 
 A finite tensor category is called {\it{pointed}} if all its simple objects are invertible. It is thus easy to see that any finite tensor category which is pointed is equivalent
 to $Vect(G, \omega)$ for some finite group $G$ and some 3-cocycle $\omega$.

 \subsection{Module Categories} Following \cite[\S 2.3]{Ost}, a right {\it{module category}} over the tensor category $(\CC, \otimes, 1_\CC, \alpha, \lambda, \rho)$ consists of $(\MM, \otimes , \mu, \tau)$ where $\MM$ is a category, $\otimes : \MM \times \CC \to \MM$  is an exact
 bifunctor, $\mu_{M,X,Y}: M \otimes (X\otimes Y) \stackrel{\sim}{\to}
 (M \otimes X) \otimes Y$ is a functorial associativity and $\tau_M: M \otimes 1_\CC \stackrel{\sim}{\to} M$ is a unit isomorphism for any $X,Y \in \CC$, $M \in \MM$, satisfying the pentagon axiom
\begin{align}\xymatrix{
& M \otimes ((X \otimes Y) \otimes Z) \ar[ld]_{id_M  \otimes \alpha_{X,Y,Z}} \ar[dr]^{\mu_{M,X \otimes Y, Z}} & \\
 M \otimes (X \otimes (Y \otimes Z)) \ar[d]^{\mu_{M,X,Y \otimes Z}} && ( M \otimes (X \otimes Y)) \otimes Z\ar[d]^{\mu_{M,X,Y} \otimes id_Z}\\
 (M \otimes X) \otimes ( Y \otimes Z) \ar[rr]_{\mu_{M \otimes X,Y,Z}} &&  ((M \otimes X) \otimes Y) \otimes Z
} \label{pentagon axiom module category}\end{align}
 and the triangle axiom 
\begin{align}\xymatrix{M \otimes (1_\CC \otimes Y) \ar[dr]_{id_M \otimes \lambda_Y}
\ar[rr]^{\mu_{M, 1_\CC, Y}} && (M \otimes 1_\CC) \otimes Y \ar[ld]^{\tau_M \otimes id_Y}\\
& M \otimes Y. &
}\label{triangle axiom module category}\end{align}

A {\it{module functor}} $(F,\gamma): (\MM_1 , \mu^1, \tau^1)\to (\MM_2 , \mu^2, \tau^2)$ 
between two module categories consist of a functor $F : \MM_1 \to \MM_2$ and a functorial isomorphism
$\gamma_{M,X}: F(M \otimes X) \to F(M)\otimes X$ for any $X \in \CC$, $M \in \MM$, satisfying the pentagon
axiom
$$\xymatrix{
&F( M \otimes (X \otimes Y)) \ar[ld]_{F(\mu^1_{M,X,Y})} \ar[dr]^{\gamma_{M,X \otimes Y}} & \\
 F(( M \otimes X) \otimes Y) \ar[d]^{\gamma_{M \otimes X, Y}} &&F( M )\otimes (X \otimes Y)\ar[d]^{\mu^2_{F(M),X,Y} }\\
 F( M \otimes X) \otimes Y \ar[rr]_{\gamma_{M, X \otimes id_Y}} && (F( M ) \otimes X) \otimes Y
}$$
 and the triangle axiom 
$$\xymatrix{F(M \otimes 1_\CC) \ar[dr]_{\gamma_{M, 1_\CC}}
\ar[rr]^{F(\tau^1_M)} && F(M) \\ 
& F(M) \otimes 1_\CC. \ar[ru]_{\tau^1_{F(M)}}  &
}$$
Two module categories $\MM_1$ and $\MM_2$ over $\CC$ are {\it{equivalent}} if there exist a module functor between the two which
is moreover an equivalence of categories. The {\it{direct sum}} $\MM_1 \oplus \MM_2$ is the module category with the obvious structure. A module
category is {\it{indecomposable}} if it is not equivalent to the direct sum of
two non-trivial module categories.

   A {\it{natural module transformation}} $\eta: (F^1, \gamma^1) \to (F^2, \gamma^2)$ consist of a natural transformation $\eta: F^1 \to F^2$
   such that the square
   $$\xymatrix{
   F^1(M \otimes X) \ar[r]^{\eta_{M \otimes X }}  \ar[d]_{ \gamma^1_{M, X}} &  F^2(M \otimes X) \ar[d]^{\gamma^2_{M,X}}\\
   F^1(M) \otimes X \ar[r]_{\eta_M \otimes id_X} & F^2(M) \otimes X.
   }$$
   commutes for all $M \in \MM$ and $X \in \CC$.
   
   \subsection{Indecomposable module categories over $\VV(G, \omega)$}
   \label{subsection indecomposable module categories}
   Let $\MM$ be a skeletal right module category over $\VV(G, \omega)$. The set of simple objects of $\MM$ is a transitive right $G$-set and therefore it
   can be identified with the coset $K :=A \backslash G$ for $A$ a subgroup of $G$. The isomorphisms $\mu_{k,g_1,g_2}$ for $k \in K$ and $g_1,g_2 \in G$ are scalars, and we can assemble these scalars as an element
   $$\mu \in C^2(G, \Map(K, \Cx)), \ \ \mu(k;g_1,g_2):= \mu_{k,g_1,g_2}.$$
   The pentagon axiom \eqref{pentagon axiom module category} translates into the equation
   $$\omega(g_1,g_2,g_3) \mu(k;g_1,g_2g_3) \mu(k \trl g_1; g_2, g_3)= \mu(k; g_1g_2,g_3)\mu(k;g_1,g_2),$$
   which in view of the definition of the differential $\delta_G$ in \eqref{differential G} becomes
   \begin{align} \label{delta mu = omega}
   \delta_G \mu^{-1} = \pi^*\omega
   \end{align}
   where $\pi^* \omega \in C^3(G,\Map(K, \Cx))^K$ is the $K$-invariant cocycle defined by $\omega$, i.e.
   $$\pi^* \omega (k;g_1,g_2,g_3) : = \omega (g_1,g_2,g_3).$$
   
   Since $\mu$ is normalized and the unit constraint in $\VV(G, \omega)$ is trivial, we have that the triangle axiom \eqref{triangle axiom module category} implies that the unit constraint in $\MM$ is trivial.
   
   Denote this skeletal module category $\MM = \MM(A \backslash G, \mu)$. Note that two $\VV(G, \omega)$-module categories
   $\MM_1 = \MM(A_1 \backslash G, \mu_1)$ and 
   $\MM_2 = \MM(A_2 \backslash G, \mu_2)$ are equivalent if
   and only if there exist a right $G$-equivariant isomorphism   
   $F: A_1 \backslash G \stackrel{\cong}{\to} A_2 \backslash G$
      and an element $\gamma \in C^1(G, \Map(A_1 \backslash G, \Cx))$ such that
      $$\gamma(A_1g ; g_1g_2) \mu_2(F(A_1g);g_1,g_2)= \mu_1(A_1g;g_1,g_2) \gamma(A_1gg_1;g_2) \gamma(A_1g;g_1).$$  
   This information implies that $A_1$ and $A_2$ are conjugate
   subgroups of $G$ and that 
$$   \delta_G \gamma = \frac{F^*\mu_2}{\mu_1}.$$

In the case that $A=A_1=A_2$, the $G$-equivariant isomorphisms
are parameterized by the elements of the group $A \backslash N_G(A)$, and the equation $\delta_G \gamma = \frac{F^*\mu_2}{\mu_1}$
    implies that $\frac{F^*\mu_2}{\mu_1}$ is trivial in $H^2(G,\Map(A \backslash G, \Cx))$. Since we know that $\widetilde{\psi}:H^2(G,\Map(A \backslash G, \Cx)) \stackrel{\cong}{\to} H^2(A, \Cx)$ is an isomorphism, we can
    conclude that the isomorphism classes of module categories over
    $\VV(G, \omega)$ may be parameterized (in a non-canonical manner) by pairs $([A],[\psi(\mu)])$
    where $[A]$ is a conjugacy class of subgroups of $G$, and $[\psi(\mu)]$ is a representative of a cohomology class in the group of invariants
    $H^2(A, \Cx)/{N_G(A)}$.

 \subsection{Dual category} \label{subsection Dual category} Let $\CC$ be a tensor category and $\MM$ an indecomposable right module category. The dual category
 $\CC^*_\MM := \Fun_\CC(\MM,\MM)$ is the category whose objects are module
 functors from $\MM$ to itself and whose morphisms are natural module transformations.
 
 The category $\CC_\MM^*$ becomes a tensor category by composition of functors, namely for 
 $(\gamma^1, F_1), (\gamma^2, F_2) \in \Obj(\CC_\MM^*)$ where $\gamma^1, \gamma^2$
  represent the module structures on the functors $F_1$ and $F_2$
 respectively, we define the tensor structure by
 $(\gamma^1, F_1)\otimes (\gamma^2, F_2):= (\gamma, F_1 \circ F_2)$ where the module structure $\gamma$ is defined by $\gamma_{M,X}:=\gamma^1_{F_2(M),X} \circ F_1(\gamma^2_{M,X})$ for $M \in \MM$ and $X \in \CC$. For two morphisms $\eta : (\gamma^1, F_1) \to (\gamma^2, F_2)$ and
$\eta' : (\gamma'^1, F'_1) \to (\gamma'^2, F'_2)$ in $\CC_\MM^*$ their tensor product is 
$(\eta \otimes \eta' )(M): = \eta_{F_2'(M)} \circ F_1(\eta'_M)$.
 
 Whenever $\CC$ and $\MM$ are semisimple, the dual category $\CC_\MM^*$ is semisimple \cite[\S 2.2]{Ost-2}. Moreover, since $\MM$ is itself a left module category over $\CC_\MM^*$ it has been shown 
in  \cite[Cor. 4.1]{Ost} that the double dual is tensor equivalent to the original category, i.e.  $(\CC^*_\MM)^*_\MM \simeq \CC$. Furthermore, the module categories of $\CC$ and of $\CC_\MM^*$ are
 in canonical bijection \cite[Prop. 2.1]{Ost-2} by the following maps.
  For $\MM_1$ a module category over $\CC$, the category $\Fun_{\CC}(\MM_1,\MM)$ of module functors from $\MM_1$
  to $\MM$ is a left module category of $\CC_\MM^*=\Fun_{\CC}(\MM,\MM)$ via the composition of functors. Conversely, if $\MM_2$
  is a left module category over $\CC_\MM^*$, then  
  $\Fun_{\CC_\MM^*}(\MM,\MM_2)$ is a right module category
  over $\Fun_{\CC_\MM^*}(\MM,\MM)=(\CC_\MM^*)_\MM^* \simeq \CC$ via composition of functors. These maps are inverse from each other.
 
 \subsection{Center of a tensor category} The {\it{center}} $\ZZ(\CC)$
 of the tensor category $\CC$  is the category whose objects are pairs $(X, \eta)$ where $X$ is an object in $\CC$ and $\eta$
 is a functorial set of isomorphisms $\eta_Y : X \otimes Y \to Y \otimes X$ such that the hexagon diagram
 $$\xymatrix{
 (X \otimes Y) \otimes Z \ar[r]^{\alpha} \ar[d]^{\eta_Y \otimes 1}& X \otimes (Y \otimes Z) \ar[r]^{\eta_{Y \otimes Z}} & (Y \otimes Z) \otimes X \ar[d]^\alpha \\
(Y \otimes X) \otimes Z \ar[r]^\alpha & Y \otimes (X \otimes Z) \ar[r]^{1 \otimes \eta_{Z}} & Y \otimes ( Z \otimes X)
  }$$
 and the triangle diagram
 $$\xymatrix{
 X \otimes 1_\CC \ar[rr]^{\eta_{1_\CC}} \ar[rd]_\rho && 1_\CC \otimes X \ar[dl]^\lambda\\
 & X &
 }$$
 are commutative. A morphism $f:(X, \eta) \to (Y , \nu)$ consists of a morphism $f:X \to Y$ for which the diagram 
 $$\xymatrix{
 X \otimes Z \ar[r]^{\eta_Z} \ar[d]_{f \otimes 1} & Z \otimes X \ar[d]^{1 \otimes f} \\
 Y \otimes Z \ar[r]_{\nu_Z} & Z \otimes Y 
 }$$
 commutes for any object $Z$ in $\CC$.  The tensor structure is defined as $(X , \eta) \otimes (Y, \nu): = (X \otimes Y, \gamma)$
 where $\gamma_Z$ is defined as the composition
 $$\xymatrix{(X \otimes Y) \otimes Z  \ar[r]^\alpha &
 X \otimes (Y \otimes Z) \ar[r]^{1 \otimes \nu_Z} &
 X \otimes (Z \otimes Y ) \ar[dl]_{\alpha^{-1}} & \\
& (X \otimes Z ) \otimes Y \ar[r]^{\eta_Z \otimes 1} &
 (Z \otimes X) \otimes Y \ar[r]^{\alpha} &
 Z \otimes (X \otimes Y).
  }$$
  
  The center $\ZZ(\CC)$ is moreover {\it{braided}} and the braiding for the pair $(X, \eta), (Y , \nu)$ is precisely 
  the map $\eta_Y$.
  
  The center $\ZZ(Vect(G, \omega))$ of the tensor category
  $Vect(G, \omega)$ contains the information necessary for constructing the quasi-Hopf algebra that is known as the Twisted
  Drinfeld Double $D^\omega(G)$ of the group $G$ twisted by $\omega$ (see \cite[\S 3.2]{Dijkgraaf}).
 
 \subsection{Weak Morita equivalence of tensor categories}
 \label{subsection Weak Morita equivalence of tensor categories}
 
 Two tensor categories $\CC$ and $\DD$ are {\it{weakly Morita equivalent}} if there exists an
  indecomposable right module category $\MM$ over $\CC$ such that $\CC_\MM^*$ and $\DD$
   are tensor equivalent \cite[Def 4.2]{MugerI}. In \cite[Prop. 4.6]{MugerI} it is shown that weak
    Morita equivalence is an equivalence relation, and
     in \cite[Thm. 3.1]{W-GT-fusion-cat} it is shown that two tensor categories are weak Morita
      equivalent if and only if their centers are braided equivalent.
 In particular we have that for $\MM$ an indecomposable module
 category over $\CC$ there is a canonical equivalence of braided tensor categories
 $\ZZ(\CC) \simeq \ZZ(\CC_\MM^*)$ \cite[Prop. 2.2]{Ost-2}.
 
  \section{The dual of $\VV(G, \omega)$ with respect to $\MM(A \backslash G, \mu)$}
 Let us consider the tensor category $\CC = \VV(G, \omega)$ and the right module
 category $\MM=\MM(A \backslash G, \mu)$ described in \S \ref{subsection indecomposable module categories}. In this chapter we will review the main results of \cite {Naidu} where explicit conditions
 are stated under which the dual category $\CC_\MM^*$ is pointed. For the sake of completeness and clarity we will review the constructions done in \S 3 and \S 4 of \cite{Naidu} and we will
 reinterpret the equations  given there in the terminology that we have set up in \S \ref{subsection Abelian group extensions} and \S \ref{subsection Cohomology of groups and the Lyndon-Hochschild-Serre spectral sequence}.
  
 \subsection{Conditions for $\CC_\MM^*$ to be pointed} Let us setup some notation for this section:
let  $K := A \backslash G$, $u : K \to G$ satisfy $p \circ u =1_G$ and $u(p(1_G))=1_G$ for $p:G \to K$ the projection, $\kappa: K \times G \to A$ satisfy  $u(k)g = \kappa_{k,g} u(k \trl g)$
and $K^A$ the elements of $K$ fixed under the conjugation by elements of $A$.
 The module category $\MM(A \backslash G, \mu)$ is the skeletal category whose simple objects
 are the elements of $K=A \backslash G$, whose tensor structure is $k \otimes g := k \trl g$ for $k\in K$ and $g \in G$ and whose associativity constraint $\mu$ satisfies $\delta_G \mu^{-1} = \pi^* \omega$, see \eqref{delta mu = omega}
In what follows  we will focus on parametrizing the
invertible objects of $\CC_\MM^*$.
 
 Following \cite[Lemma 3.2]{Naidu} any invertible module functor in $\CC_\MM^*$ is of the form $(F_y, \gamma)$
 where the functor $F_y: \MM \to \MM$ is the one that extends the $G$-equivariant map $f_y : K \to K$,   $f_y(k)=p(u(y)u(k))$ for $y \in K^A$, and $\gamma$ is a functorial isomorphism 
 $\gamma_{k,g} : F_y(k \otimes g) \stackrel{\cong}{\to} F_y(k) \otimes g$
 that satisfies the pentagon axiom. Writing $\gamma_{k,g} := \gamma(k;g) id_{p(u(y)u(k \trl g))}$
for $\gamma \in C^1(G, \Map(K, \Cx))$ we have that the pentagon axiom of a module functor translates into the equation
$$ \mu(k;g_1,g_2) \gamma(k \trl g_1; g_2) \gamma(k;g_1) = \gamma(k; g_1g_2) \mu(f_y(k); g_1,g_2)$$
which can also be written as
$$\delta_G \gamma (k; \gamma_1,\gamma_2) = \frac{\mu(f_y(k); g_1,g_2)}{\mu(k;g_1,g_2) }.$$
The inverse of $(F_y, \gamma)$ is the module functor $(F_{p(u(y)^{-1})}, \bar{\gamma})$
with $$\bar{\gamma}(k; g):= \gamma(p(u(y)^{-1}u(x))^{-1}; g)^{-1}.$$

Defining for each $ y \in K^A$ the set
$$Fun_y := \left\{ \gamma \in C^1(G, \Map(K,\Cx)) | \delta_G \gamma(k; g_1,g_2) =\frac{\mu(f_y(k); g_1,g_2)}{\mu(k;g_1,g_2) } \right\}$$
for all $k \in K$ and $g_1, g_2 \in G$, we have that the set of invertible objects of $\CC_\MM^*$ are precisely the module functors $(F_y, \gamma)$ where $y \in K^A$ and $\gamma \in Fun_y$. To simplify the notation we will denote such module functor by the pair $(y,\gamma)$.

 Two invertible module functors $(y_1, \gamma^1)$ and $(y_2, \gamma^2)$ in $\CC_\MM^*$  
 are isomorphic if and only if $y_1=y_2$ and if there exists natural transformation parameterized by a map $\eta \in C^0(G,\Map(K,\Cx))$ satisfying the equation
 \begin{align}
 \gamma^1(k;g) \eta(k)=\eta(k \trl g) \gamma^2(k;g) \label{equation isomorphic module functors}
 \end{align}
 for all $k \in K$ and $g \in G$. These equations can be rewritten as the equation
 $$\delta_G \eta = \frac{\gamma^2}{\gamma^1}$$
 in $C^1(G, \Map(K,\Cx))$. Therefore for each $y \in K^A$ we may define an equivalence
 relation on the elements $\gamma^1, \gamma^2 \in Fun_y$ by setting that $\gamma^2 \simeq \gamma^1$ whenever there exist $\eta$ such that $\delta_G \eta = \frac{\gamma^2}{\gamma^1}$; denote the by $\overline{Fun}_y$ the associated set of equivalence classes.  

For each $y  \in K^A$ let us choose an element $\gamma_y \in Fun_y$, and note that the maps
$$Fun_y \to Z^1(G, \Map(K,\Cx)), \ \beta \mapsto \frac{\beta}{\gamma_y}, \ \ \ \ Z^1(G, \Map(k,\Cx)) \to Fun_y, \ \epsilon \mapsto \epsilon \gamma_y$$
are inverse to each other. Therefore we obtain bijections
$$\overline{Fun}_y \cong H^1(G, \Map(K, \Cx)) \cong H^1(A, \Cx) ={{\mathbb{A}}}$$
which are realized by the maps
\begin{align}
\zeta_y : &{{\mathbb{A}}} \to Fun_y, & \zeta_y(\rho):=& \gamma_y \varphi(\rho) \nonumber\\
\theta_y: &Fun_y \to {{\mathbb{A}}}, & \theta_y(\beta) :=& \psi(\beta / \gamma_y). \label{maps from widehatA to Funy}
\end{align}

  Recall from \cite[Def. 2.2]{ENO} that the {\it{ global dimension}} $dim(\CC)$ of a fusion 
 category $\CC$ is the sum of the squared norms of its simple objects, and note that by \cite[Thm. 2.15 ]{ENO} we have  $dim(\CC_\MM^*)=dim(\CC)$ whenever $\CC$ is a fusion category and $\MM$ is an indecomposable module category over $\CC$. 
 
  Let us suppose now that the dual category $\CC_\MM^*=\VV(G,\omega)_{\MM(A \backslash G, \mu)}^*$
  is pointed. Therefore its global dimension 
  $$dim(\CC_\MM^*)= | {{\mathbb{A}}}||K^A|$$ must be equal to the number
  of isomorphic classes of invertible objects, since on pointed categories all simple objects are invertible. On the other hand, by \cite[Thm. 2.15 ]{ENO} we have that
  $dim(\CC_\MM^*)=dim(\CC)$ and $dim(\CC)=|G|$. Therefore in order for the category
  $\CC_\MM^*$ to be pointed it is necessary that $| {{\mathbb{A}}}||K^A|=|G|$. Since $|G|=|A||K|$, $|{{\mathbb{A}}}| \leq |A|  $ and $|K^A| \leq |K|$, the equality holds if and only if $A$ is abelian, thus having that $|{{\mathbb{A}}}|= |A|$, and if $A$ is normal in $G$ and $K^A=K$.
  
  On the other hand, if $A$ is abelian and normal on $G$, then the number of isomorphism classes of
  invertible objects in $\CC_\MM^*$ is $|{{\mathbb{A}}}||K|=|G|$. Since $dim(\CC_\MM^*)=dim(\CC)=|G|$ then we have that the set of isomorphism classes of invertible objects exhaust the set of simple elements, and therefore $\CC_\MM^*$ must be pointed.
  
  Summarizing we have
  \begin{theorem} \cite[Thm 3.4]{Naidu}  \label{Theorem conditions for pointed}
  The tensor category $\CC_\MM^*=\VV(G,\omega)_{\MM(A \backslash G, \mu)}^*$ is pointed if and only if $A$ is abelian and normal in $G$ and the cohomology class $[\frac{\mu \trl y}{\mu}]$ is trivial in $H^2(G, \Map(K, \Cx))$ for all $y \in K$.

    \end{theorem}

  Note that since $A$ is normal in $G$, we may use the notation introduced in \S \ref{subsection Cohomology of groups and the Lyndon-Hochschild-Serre spectral sequence} so that $\mu(f_y(k);g_1,g_2)=\mu(yk;g_1,g_2)= (\mu \trl y)(k;g_1,g_2).$
Since we have that $\delta_G \mu^{-1} = \pi^* \omega = \delta_G (\mu^{-1} \trl y)$, the quotient $\frac{\mu \trl y}{\mu}$ defines a cocycle in $Z^2(G, \Map(K, \Cx))$. The equation $\delta_G \gamma_y  = \frac{ \mu \trl y}{\mu }$ implies that the quotient is trivial in cohomology.

  \subsection{The Grothendieck ring of the pointed category  $\CC_\MM^*$} \label{Subsection The Grothendieck ring of the pointed category}
  From now on we will assume that the dual category $\CC_\MM^*$ is pointed. Therefore we 
  have that $A$ is abelian and normal in $G$ and that we can choose elements $\gamma_y \in 
  C^1(G, \Map(K, \Cx))$ for each $y \in K$ such that $\delta_G \gamma_y = \frac{\mu \trl y}{\mu}$.
  
  The Grothendieck ring $K_0(\CC_\MM^*)$ of the category $\CC_\MM^*$ is the ring defined by the semi-ring whose elements are the isomorphism classes of objects and whose product is the one induced by the tensor product. Since
  $\CC_\MM^*$ is pointed then $K_0(\CC_\MM^*)$ is isomorphic to the group ring $\integer[\Lambda]$ for some finite group $\Lambda$. In this section we will recall the construction of this isomorphism carried out in \cite[Thm. 4.5]{Naidu}.
  
  The tensor product of two invertible elements $(y_1,\gamma^1)$,  $(y_2,\gamma^2)$ in $\CC_\MM^*$ as defined in \S \ref{subsection Dual category} is
  $$(y_1,\gamma^1) \otimes (y_2,\gamma^2) = (y_1y_2,(\gamma^1 \trl y_2)\gamma^2).$$
  This tensor product defines a group structure on the set of isomorphism classes of invertible objects
  $$\Lambda:= \bigcup_{y \in K } \{y\} \times \overline{Fun}_y $$ 
  by the equation $(y_1,[\gamma^1]) \star (y_2,[\gamma^2]) = (y_1y_2, [(\gamma^1 \trl y_2)\gamma^2])$ where $[\gamma]$ denotes the equivalence class of $\gamma$ in $Fun_y$.
  
  Define the element $\gamma \in C^1(K, C^1(G, \Map(K, \Cx)))$ by the equation
  $$\gamma(y):=\gamma_y$$ and note that the equations $\delta_G \gamma_y = \frac{\mu \trl y}{\mu}$
  are equivalent to the equation
  $$\delta_G \gamma = \delta_K \mu.$$
  Define the element $\tilde{\nu} := \delta_K \gamma$, i.e. $\tilde{\nu}(y_1,y_2)= \frac{\gamma(y_2) \gamma(y_1)\trl y_2}{\gamma(y_1y_2)}$, and note that 
  $$\delta_K \tilde{\nu} = \delta_K^2 \gamma = 1 \ \ \mbox{and} \ \ 
\delta_G \tilde{\nu} = \delta_G \delta_K \gamma= \delta_K \delta_G \gamma = \delta_K^2 \mu=1.$$  
  Hence $\tilde{\nu} \in Z^2(K, Z^1(G,\Map(K,\Cx)))$ and we may define
  \begin{align} \label{definition nu}{\nu} := \psi \circ \tilde{\nu} \in Z^2(K, Z^1(A,\Cx))= Z^2(K, {{\mathbb{A}}})
  \end{align}
  thus having ${\nu}(y_1,y_2)(a):= \tilde{\nu}(y_1,y_2)(1;a)$.
  
  With this 2-cocycle ${\nu}$ we may define the crossed product $K \ltimes_{{\nu}} {{\mathbb{A}}}$
  by setting on pairs of elements of the set $K \times {{\mathbb{A}}}$ 
  $$(y_1, \rho_1) \cdot (y_2, \rho_2) := (y_1y_2, \rho_1^{y_2} \rho_2 {{\nu}}(y_1,y_2)).$$
  
 Using the notation of \eqref{maps from widehatA to Funy} we have
  
  \begin{theorem} \cite[Thm. 4.5]{Naidu}
  The map $$T : K \ltimes_{{\nu}}{{\mathbb{A}}} \to \Lambda, \ \ \ T((y, \rho))= (y,[\zeta_y(\rho)])$$
  is an isomorphism of groups. Hence $K_0(\CC_\MM^*) \cong \integer [ K \ltimes_{{\nu}} {{\mathbb{A}}}]$.
  \end{theorem}
  \begin{proof}On the one hand we have
  \begin{align*}
  T((y_1, \rho_1) \cdot (y_2,\rho_2))=&T( (y_1y_2, \rho_1^{y_2} \rho_2 {{\nu}}(y_1,y_2))) \\ = &(y_1y_2,[\zeta_{y_1y_2}(\rho_1^{y_2} \rho_2 {{\nu}}(y_1,y_2))])
  \end{align*}
  and on the other
  \begin{align*}
  T((y_1, \rho_1)) \star T((y_2,\rho_2)) =& (y_1,[\zeta_{y_1}(\rho_1)]) \star (y_2,[\zeta_{y_2}(\rho_2)]) \\
  =& (y_1y_2, [( \zeta_{y_1}(\rho_1) \trl y_2) \zeta_{y_2}(\rho_2)])
  \end{align*}
  The result follows if we check the equality $$\theta_{y_1y_2}((\zeta_{y_1}(\rho_1) \trl y_2) \zeta_{y_2}(\rho_2)) = \rho_1^{y_2} \rho_2 {{\nu}}(y_1,y_2)$$
  since this implies that $\zeta_{y_1y_2}((\rho_1 \trl y_2) \rho_2 {{\nu}}(y_1,y_2))$ and $(\zeta_{y_1}(\rho_1) \trl y_2) \zeta_{y_2}(\rho_2)$ are cohomologous; hence we have
 
  \begin{align*}
  \theta_{y_1y_2}((\zeta_{y_1}(\rho_1) \trl y_2) \zeta_{y_2}(\rho_2))(a) =&  \frac{((\zeta_{y_1}(\rho_1) \trl y_2) (1;a)) \zeta_{y_2}(\rho_2) (1;a)}{\gamma(y_1y_2)(1;a)}\\
  =& \frac{(\gamma(y_1)\trl y_2 \varphi(\rho_1) \trl y_2)(1;a) (\gamma(y_2) \varphi(\rho_2))(1;a)}{\gamma(y_1y_2)(1;a)}\\
  =& \delta_K \gamma (y_1,y_2)(1;a) \rho_1^{y_2}(a) \rho_2(a)\\
  =& ({{\nu}}(y_1,y_2) \rho_1^{y_2} \rho_2)(a).
  \end{align*}
  \end{proof}
  
    \subsection{A skeleton of the pointed category  $\CC_\MM^*$}
  A skeleton $sk(\CC_\MM^*)$ of $\CC_\MM^*$ is a full subcategory of $\CC_\MM^*$ on which each object of 
  $\CC_\MM^*$ is isomorphic to only one object in $sk(\CC_\MM^*)$. Let us choose for objects
  $$ob(sk(\CC_\MM^*)):= \{(y, \zeta_y(\rho)) | (y,\rho) \in K \ltimes_\nu {{\mathbb{A}}} \}$$  
  and define its tensor product $\bullet$ by the one induced by $\star$, i.e.
  $$ ((y_1, \zeta_{y_1}(\rho_1)) \bullet (y_2,\zeta_{y_2}(\rho_2)) : = (y_1y_2, \zeta_{y_1y_2}({{\nu}}(y_1,y_2) \rho_1^{y_2}\rho_1)).$$
  
  For each pair of objects, choose isomorphisms in $\CC_\MM^*$
  \begin{align*}
  f((y_1, \zeta_{y_1}(\rho_1)),  (y_2,\zeta_{y_2}(\rho_2)) : (y_1, \zeta_{y_1}(\rho_1)) \bullet (y_2,\zeta_{y_2}(\rho_2) )\stackrel{\sim}{\to} (y_1, \zeta_{y_1}(\rho_1)) \tensor (y_2,\zeta_{y_2}(\rho_2) )
  \end{align*}
  which by equation \eqref{equation isomorphic module functors} satisfy
  \begin{align*}
 (( \zeta_{y_1}(\rho_1) \trl y_2) \zeta_{y_1}(\rho_1)) (k;g) = \frac{ f((y_1, \zeta_{y_1}(\rho_1)),  (y_2,\zeta_{y_2}(\rho_2))(k \trl g)}{ f((y_1, \zeta_{y_1}(\rho_1)),  (y_2,\zeta_{y_2}(\rho_2))(k)}&\\ \times \zeta_{y_1y_2}&({{\nu}}(y_1 ,y_2) \rho_1^{y_2}\rho_1)(k;g).
  \end{align*}
  
  The tensor product $\otimes$ in $\CC_\MM^*$ is associative since it is defined by the composition of functors, but the tensor product $\bullet$ in its skeleton $sk(\CC_\MM^*)$ may fail to be associative.
  The associativity constraint for $sk(\CC_\MM^*)$ is then
  \begin{align*}
  \widehat{\omega}'((y_1, \zeta_{y_1}(\rho_1)),  (y_2,\zeta_{y_2}(\rho_2)),(y_3 &,\zeta_{y_3}(\rho_3))) \\
  =& \frac{f((y_1, \zeta_{y_1}(\rho_1)),  (y_2,\zeta_{y_2}(\rho_2)) \otimes Id_{(\zeta_{y_3}(\rho_3),y_3)}}{f((y_1, \zeta_{y_1}(\rho_1)),  (y_2,\zeta_{y_2}(\rho_2))\bullet (y_3,\zeta_{y_3}(\rho_3)))} \\ &\times  \frac{f((y_1, \zeta_{y_1}(\rho_1))\bullet  (y_2,\zeta_{y_2}(\rho_2)),(y_3,\zeta_{y_3}(\rho_3)))}{ Id_{(\zeta_{y_1}(\rho_1),y_1)} \otimes f((y_2,\zeta_{y_2}(\rho_2)),(y_3,\zeta_{y_3}(\rho_3))). }
  \end{align*}
  
  In \cite[Thm. 4.9]{Naidu} it is shown that $\widehat{\omega}'$ is $K$-invariant and moreover that it can be given in explicit form by the equation
   \begin{align*}
     \widehat{\omega}'((y_1, \zeta_{y_1}(\rho_1)),  (y_2,\zeta_{y_2}(\rho_2)),(y_3 &,\zeta_{y_3}(\rho_3))) = \tilde{\nu}(y_1,y_2)(1;u(y_3)) \ \rho_1(\kappa_{y_2,u(y_3)}).
   \end{align*}
   Therefore we may define the 3-cocycle on $K \ltimes_{{\nu}} {{\mathbb{A}}}$ by the equation
    \begin{align*}  \widehat{\omega}((y_1, \rho_1),  (y_2,\rho_2),(y_3 ,\rho_3)) = \tilde{\nu}(y_1,y_2)(1;u(y_3)) \ \rho_1(\kappa_{y_2,u(y_3)}),
    \end{align*}
  and choosing $G = A \rtimes_F K$ and $u(y)=(1,y)$ as it was done at the end of \S \ref{subsection Abelian group extensions}, the 3-cocycle on $K \ltimes_{{\nu}} {{\mathbb{A}}}$ becomes
    \begin{align}  \label{definition widehat omega}
     \widehat{\omega}((y_1, \rho_1),  (y_2,\rho_2),(y_3 ,\rho_3)) = \tilde{\nu}(y_1,y_2)(1;(1,y_3)) \ \rho_1(F(y_2,y_3)).
   \end{align}
   Therefore the skeleton $sk(\CC_\MM^*)$ of $\CC_\MM^*$ becomes isomorphic to $\VV(K \ltimes_{{\nu}} {{\mathbb{A}}}, \widehat{\omega})$ which is equivalent to $Vect(K \ltimes_{{\nu}} {{\mathbb{A}}}, \widehat{\omega})$. Therefore we can conclude with
   \begin{theorem} \cite[Thm. 4.9]{Naidu}
   The fusion categories $\CC_\MM^*=\VV(G,\omega)_{\MM(A \backslash G, \mu)}^*$ and
   $Vect(K \ltimes_{{\nu}} {{\mathbb{A}}}, \widehat{\omega})$ are equivalent.
   \end{theorem}
   
   Applying the results of \S \ref{subsection Weak Morita equivalence of tensor categories} we have
   \begin{cor} \label{corollary weak Morita equivalence}
   The categories $Vect( A \rtimes_F K, \omega)$ and $Vect(K \ltimes_{{\nu}} {{\mathbb{A}}}, \widehat{\omega})$ are
   weakly Morita equivalent. Hence their centers are are canonically equivalent 
   $$\ZZ(Vect( A \rtimes_F K, \omega)) \simeq \ZZ(Vect(K \ltimes_{{\nu}} {{\mathbb{A}}}, \widehat{\omega}))$$
   as braided tensor categories
   \end{cor}
   
   \section{Weak Morita equivalence classes of group-theoretical tensor categories}
   
We are interested in classifying group theoretical tensor categories of a specific global dimension up to weak Morita equivalence. For this purpose we will fix the group $G= A \rtimes_F K$ with $A$ abelian and normal in $G$ and $F \in Z^2(K,A)$, and we will give an explicit description of the cocycles
$\omega \in Z^3(A \rtimes_F K, \Cx)$ and $\widehat{\omega} \in Z^3(K \ltimes_{{\nu}} {{\mathbb{A}}},  \Cx)$
   such that the tensor categories $\VV( A \rtimes_F K, \omega)$ and $\VV( K \ltimes_{{\nu}} {{\mathbb{A}}}, \widehat{\omega})$ are weakly Morita equivalent.
   
   \subsection{Description of $\omega$, $\mu$ and $\gamma$} \label{subsection omega nu}
   
   In Theorem \ref{Theorem conditions for pointed} and in \S \ref{Subsection The Grothendieck ring of the pointed category}  we have seen the conditions needed for the 
   tensor category $\CC_\MM^*=\VV(G,\omega)_{\MM(A \backslash G, \mu)}^*$ to be pointed. In particular we have seen that we need
   the existence of $\gamma \in C^1(K, C^1(G, \Map(K, \Cx)))$ such that 
   $$\delta_G \gamma= \delta_K \mu.$$
   Since we also have that $\delta_G \mu^{-1}= \pi^* \omega$ we can obtain the following lemma.
   
   \begin{lemma} \label{lemma omega cohomologous to tilde nu}
     The cocycles $\pi^* \omega$ and $\tilde{\nu}$ are cohomologous in $\Tot(C^*(K, C^*(G,\Map(K, \Cx))))$. 
   \end{lemma}
 
   \begin{proof}
Recall the definition of the double complex
$C^*(K, C^*(G,\Map(K, \Cx)))$
given in \S \ref{subsubsection Double complex}, and note
that we have $\pi^* \omega \in C^{0,3}$, $\mu \in C^{0,2}$, 
$\gamma \in  C^{1,1}$ and $\tilde{\nu} = \delta_K \gamma \in C^{2,1}$, satisfying $\pi^* \omega \cdot \delta_G \mu=1$ and $ \delta_K \mu \cdot \delta_G \gamma^{-1}=1$.

Consider the element $\mu \oplus \gamma \in \Tot^2$ and note that
$$\delta_{Tot} (\mu \oplus \gamma) = (\delta_K \oplus \delta_G^{(-1)^p}) (\mu \oplus \gamma) = \delta_G \mu \oplus \delta_K \mu \cdot \delta_G \gamma^{-1} \oplus \delta_K \gamma.$$

Therefore 
$$\pi^* \omega  \cdot \delta_{Tot} (\mu \oplus \gamma) = \tilde{\nu}.$$
\end{proof}

 Lemma \ref{lemma omega cohomologous to tilde nu} implies further conditions on the cohomology class of $\omega$ for the
   tensor category $\CC_\MM^*=\VV(G,\omega)_{\MM(A \backslash G, \mu)}^*$ to be pointed.
   \begin{cor} \label{omega in 2,1 and 3,0}
     If the 
   tensor category $\CC_\MM^*=\VV(G,\omega)_{\MM(A \backslash G, \mu)}^*$ is pointed then $\omega$ is cohomologous to a cocycle that lives 
    in $C^{2,1} \oplus C^{3,0}$ of the double complex that induces the Lyndon-Hochschild-Serre spectral sequence.
   \end{cor}
  
 \begin{rem} Note that this implies that the cohomology class of $\omega$ belongs to the subgroup of 
 $H^3(G, \Cx) $ defined as
 $$\Omega(G;A) := \ker \left( \ker \left( H^3(G, \Cx) \to E^{0,3}_\infty \right) \to E^{1,2}_\infty  \right)$$
 which fits into the short exact sequence
 $$1 \to E^{3,0}_\infty \to \Omega(G;A) \to E^{2,1}_\infty \to 1.$$
   
   The cohomology classes in $\Omega(G;A)$ are the only cohomology classes such that
   $\CC_\MM^*=\VV(G,\omega)_{\MM(A \backslash G, \mu)}^*$ is pointed.
\end{rem}
      
   In what follows we will construct explicit representatives for $\omega$ and $\mu$, but for this purpose we will start by constructing explicit 3-cocyles in
  $\Tot(C^*(K, C^*(G,\Map(K, \Cx))))$ which appear in    
  $\Omega(G;A)$. Let us start by determining the second differential $d_2 : E_2^{2,1} \to E_2^{4,0}$.

  \begin{lemma}
  The second differential $d_2 : E_2^{2,1} \to E_2^{4,0}$ is isomorphic to the homomorphism
  $$H^2(K, {{\mathbb{A}}}) \to H^4(K, \Cx), \ \ [\widehat{F}] \mapsto [(\widehat{F} \wedge F)^{-1}]$$
  where $(\widehat{F} \wedge F) (k_1,k_2,k_3,k_4):= \widehat{F}(k_1,k_2)(F(k_3,k_4))$.
  \end{lemma}
   
   \begin{proof}
 First recall that \begin{align*}E_2^{2,1} &= H^2(K, H^1(G, \Map(K, \Cx))) \cong H^2(K, \Hom(A, \Cx)) = H^2(K, {{\mathbb{A}}})\\
   E_2^{4,0} & = H^4(K, H^0(G, \Map(K, \Cx))) = H^4(K, \Map(K, \Cx)^G) \cong H^4(K, \Cx). \end{align*}
   
   Take $\widehat{F} \in Z^2(K, {{\mathbb{A}}})$ and use the map $\varphi$ of Lemma \ref{lemma varphi} to lift this cocycle to $\varphi(\widehat{F}) \in C^2(K, Z^1(G,\Map(K,\Cx)))$; in coordinates:
   \begin{align*}
   \varphi(\widehat{F})(k_1,k_2)(x_1,(a_2,x_2)) &= \widehat{F}(k_1,k_2)(\kappa_{x_1,(a_2,x_2)}) = 
   \widehat{F}(k_1,k_2)({}^{x_1}a_2 F(x_1,x_2))\\
   &= \widehat{F}(k_1,k_2)({}^{x_1}a_2) \  \widehat{F}(k_1,k_2)(F(x_1,x_2)).
   \end{align*}
   Its boundary is
   \begin{align*}
   \delta_k \varphi(\widehat{F})(k_1,k_2,k_3)&(x_1, (a_2,x_2))\\ =& \widehat{F}(k_2,k_3)({}^{x_1}a_2 F(x_1,x_2))
    \widehat{F}(k_1k_2,k_3)({}^{x_1}a_2 F(x_1,x_2))^{-1}\\
   &  \widehat{F}(k_1,k_2k_3)({}^{x_1}a_2 F(x_1,x_2))
      \widehat{F}(k_1,k_2)({}^{k_3x_1}a_2 F(k_3x_1,x_2))^{-1}\\
      =& \widehat{F}(k_1,k_2)^{k_3}( F(x_1,x_2))
       \widehat{F}(k_1,k_2)( F(k_3x_1,x_2))^{-1}\\
       =&  \widehat{F}(k_1,k_2)\left( \frac{F(k_3, x_1)}{F(k_3,x_1x_2)} \right),
   \end{align*}
   and we can define $u \in C^3(K, C^0(G, \Map(K,\Cx)))$ as
   \begin{align*}
   u(k_1,k_2,k_3)(x) := \widehat{F}(k_1,k_2)(F(k_3,x)).
   \end{align*}
   On the one hand we have
\begin{align*}
  \delta_G u(k_1,k_2,k_3)(x_1,(a_2,x_2)) = & u(k_1,k_2,k_3)(x_1x_2) u(k_1,k_2,k_3)(x_1)^{-1}  \\
 =& \widehat{F}(k_1,k_2)\left( \frac{F(k_3, x_1x_2)}{F(k_3,x_1)} \right)
   \end{align*}
   and on the other
   \begin{align*}
  \delta_K u(k_1, & k_2,k_3,k_4)(x)  \\ =&  \widehat{F}(k_2,k_3)(F(k_4,x))
    \widehat{F}(k_1k_2,k_3)(F(k_4,x))^{-1}
    \widehat{F}(k_1,k_2k_3)(F(k_4,x))\\
     &\widehat{F}(k_1,k_2)(F(k_3k_4,x))^{-1}
    \widehat{F}(k_1,k_2)(F(k_3,k_4x))\\
    =&\widehat{F}(k_1,k_2)^{k_3}(F(k_4,x))\widehat{F}(k_1,k_2)(F(k_3k_4,x))^{-1}
    \widehat{F}(k_1,k_2)(F(k_3,k_4x))\\
  =&  \widehat{F}(k_1,k_2)(F(k_3,k_4)).
   \end{align*}
   Since $\delta_Gu=\delta_K \varphi(\widehat{F})$ we have that
   \begin{align*}
   \delta_{Tot} ( \varphi(\widehat{F}) \oplus u^{-1})= \delta_K \varphi(\widehat{F}) \delta_Gu \oplus \delta_k u^{-1}= (\widehat{F} \wedge F)^{-1};
   \end{align*}
   therefore $d_2 [\varphi(\widehat{F})]= [(\widehat{F} \wedge F)^{-1}]$.
   
   \end{proof}
   
   Suppose that $d_2 [\varphi(\widehat{F})]=0$, hence there exists $\epsilon \in C^3(K, \Cx)$
   such that $\delta_K \epsilon = \widehat{F} \wedge F$.
      Define $\bar{\epsilon} \in C^3(K,C^0(G, \Maps(K,\Cx)))$ by the equation
    $$\bar{\epsilon}(k_1,k_2,k_3)(x):= {\epsilon}(k_1,k_2,k_3)$$
   and note that $\delta_K \bar{\epsilon} = \widehat{F} \wedge F$ and that $\delta_G \bar{\epsilon}=1$.
   Hence the class $\varphi(\widehat{F}) \oplus \bar{\epsilon} u^{-1} \in C^{2,1} \oplus C^{3,0}$ defines a 3-cocycle in the total complex:
 \begin{align*}
 \varphi(\widehat{F}) \oplus \bar{\epsilon} u^{-1} \in Z^3 \Tot(C^*(K, C^*(G,\Map(K, \Cx)))).
 \end{align*}  
 Define $\beta \in C^2(K,C^0(G, \Maps(K,\Cx)))$ by the equation
 \begin{align*}
 \beta(k_1,k_2)(x):= \epsilon(k_1,k_2,x)
 \end{align*}
 and note that
 \begin{align}
 \delta_K \beta (k_1,k_2,k_3)(x) = &\epsilon(k_2,k_3,x) 
 \nonumber \epsilon(k_1k_2,k_3,x)^{-1} \epsilon(k_1,k_2k_3,x)  \epsilon(k_1,k_2,k_3x)^{-1} \\
 \nonumber =& \delta_K \epsilon(k_1,k_2,k_3,x)  \epsilon(k_1,k_2,k_3)^{-1}\\
 \label{equation delta K beta} =& \widehat{F}(k_1,k_2)((F(k_3,x)) \bar{\epsilon}(k_1,k_2,k_3)(x)^{-1}.
 \end{align}
 Therefore $\delta_K \beta \  \bar{\epsilon} \ u^{-1} =1$, hence we have that the class $\varphi(\widehat{F}) \delta_G \beta \in C^{2,1}$ is a 3-cocycle in the total complex and moreover that it is 
cohomologous to the class $\varphi(\widehat{F}) \oplus \bar{\epsilon} u^{-1}$, in coordinates:
\begin{align}
(\varphi(\widehat{F}) \delta_G \beta) (k_1,k_2)(x_1,(a_2,x_2)) =& \widehat{F}(k_1,k_2)({}^{x_1}a_2) \  \widehat{F}(k_1,k_2)(F(x_1,x_2)) \nonumber\\
& \epsilon(k_1,k_2,x_1x_2) \ \epsilon(k_1,k_2,x_1)^{-1}.
\label{varphi(widehat F) delta G beta}
\end{align}

 Summarizing the previous results:
 \begin{proposition} \label{proposition omega in 2,1 3,0}
 Every cohomology class which appears in $\Omega(G;A)$ can be represented by a 3-cocycle
 $\varphi(\widehat{F}) \delta_G \beta \in C^{2,1}$ with $\widehat{F} \in Z^2(K, {{\mathbb{A}}})$, $\beta(k_1,k_2)(x)=\epsilon'(k_1,k_2,x)$ and $\delta_K \epsilon' = \widehat{F} \wedge F$.
 \end{proposition}
 \begin{proof}
 Take $[\omega] \in \Omega(G;A)$ and let $[\widehat{F}] \in E^{2,1}_2$ be
 a representative of the cohomology class of the image of $[\omega]$ in $E^{2,1}_\infty$.
 Since $d_2 [\varphi(\widehat{F})]=0$ we know that
 the cohomology class  $[ \varphi(\widehat{F}) \oplus \bar{\epsilon} u^{-1}]$ constructed above belongs
 to $\Omega(G;A)$. Therefore we have that
 $$[\omega^{-1}] \cdot [ \varphi(\widehat{F}) \oplus \bar{\epsilon} u^{-1}] \in E^{3,0}_\infty$$
 and hence we can choose a representative cocycle $[\tau] \in H^3(K,\Cx)\cong E_2^{3,0}$  
 such that
 $$[\omega] = [ \varphi(\widehat{F}) \oplus \bar{\epsilon} \ \bar{\tau} \ u^{-1}]$$
with $\bar{\tau} \in C^3(K,C^0(G, \Maps(K,\Cx)))$ defined as
 $$\bar{\tau}(k_1,k_2,k_3)(x):= {\tau}(k_1,k_2,k_3).$$
 
Let $\epsilon':= \epsilon \tau$ and define $\beta \in C^2(K,C^0(G, \Maps(K,\Cx)))$ by the equation
 \begin{align*}
 \beta(k_1,k_2)(x):= \epsilon'(k_1,k_2,x).
 \end{align*}
 Equation \eqref{equation delta K beta} implies that $\delta_K \beta = (\bar{\epsilon} \ \bar{\tau})^{-1} u$ and therefore
 the proposition follows from the equation
 $$( \varphi(\widehat{F}) \oplus \bar{\epsilon} \ \bar{\tau} u^{-1}) \delta_{Tot} \beta = \varphi(\widehat{F}) \delta_G \beta \oplus \delta_K \beta  \  \bar{\epsilon}  \ \bar{\tau} u^{-1}=\varphi(\widehat{F}) \delta_G \beta.$$
 \end{proof}

   Now we need to find an explicit description of $\omega \in Z^3(G, \Cx)$  such that $\pi^*\omega$ and $\varphi(\widehat{F}) \delta_G \beta$ are cohomologous.

  \begin{theorem} \label{theorem omega mu gamma}
 Let $G = A \rtimes_F K$ and consider $\omega \in C^3(G, \Cx)$, $ \mu \in C^{0,2}$ and $\gamma \in C^{1,1}$ defined by the following equations:
 \begin{align*}
 \omega((a_1,x_1),(a_2,x_2),(a_3,x_3)) := & \widehat{F}(x_1,x_2)(a_3) \ \epsilon(x_1,x_2,x_3)\\
 \mu(x_1, (a_2,x_2),(a_3,x_3)) =& \left(\widehat{F}(x_1,x_2)(a_3)\ \epsilon(x_1,x_2,x_3) \right)^{-1}\\
 \gamma(y)(x_1,(a_2,x_2))=&  \widehat{F}(y,x_1)(a_2)\  \epsilon(y,x_1,x_2,). 
 \end{align*}
Then $\pi^* \omega \cdot (\delta_{Tot} \mu \oplus \gamma) =\varphi(\widehat{F}) \delta_G \beta$.
\end{theorem}
\begin{proof}
Let us calculate: 
\begin{align*}
   \delta_G \mu ( x_1,&(a_2,x_2),(a_3,x_3),(a_4,x_4)) \\
 =&   \mu(x_1x_2, (a_3,x_3),(a_4,x_4)) \  \mu(x_1, (a_2{}^{x_2}a_3F(x_2,x_3),x_2x_3),(a_3,x_3))^{-1}\\
&  \mu(x_1, (a_2,x_2)(a_3{}^{x_3}a_4F(x_3,x_4),x_3x_4))
\ \mu(x_1, (a_2,x_2),(a_3,x_3))^{-1}\\
=& \widehat{F}(x_1x_2,x_3)(a_4)^{-1} \ \widehat{F}(x_1,x_2x_3)(a_4) \
\widehat{F}(x_1,x_2)(a_3{}^{x_3}a_4F(x_3,x_4))^{-1}\\
& \widehat{F}(x_1,x_2)(a_3) \ \epsilon(x_2,x_3,x_4)^{-1} \ \delta_K \epsilon(x_1,x_2,x_3,x_4)\\
=& \widehat{F}(x_2,x_3)(a_4)^{-1} \  \epsilon(x_2,x_3,x_4)^{-1},
\end{align*}
  and
  \begin{align*}
  \pi^* \omega(x_1,(a_2,x_2),(a_3,x_3),(a_4,x_4))=& \omega((a_2,x_2),(a_3,x_3),(a_4,x_4))\\ =& \widehat{F}(x_2,x_3)(a_4) \  \epsilon(x_2,x_3,x_4),
   \end{align*} 
 hence we have that $$\delta_G \mu \cdot \pi^*\omega =1.$$
 
 Now 
 \begin{align*}
 \delta_K\mu (y) ( x_1,(a_2,x_2),(a_3,x_3& )) \\
 =&   \mu(x_1, (a_2,x_2),(a_3,x_3))  \ \mu(yx_1, (a_2,x_2),(a_3,x_3))^{-1}\\
 =& \frac{\widehat{F}(yx_1,x_2)(a_3) \  \epsilon(yx_1,x_2,x_3)}{\widehat{F}(x_1,x_2)(a_3) \  \epsilon(x_1,x_2,x_3)},
 \end{align*}
 and
 \begin{align*}
 \delta_G \gamma(y&)(x_1,(a_2,x_2),(a_3,x_3)) \\
 = & \gamma(y)(x_1x_2,(a_3,x_3)) \ \gamma(y)(x_1,(a_2 {}^ {x_2}a_3F(x_2,x_3),x_2x_3))^{-1} \  \gamma(y)(x_1,(a_2,x_2)) \\
 = & \widehat{F} (y, x_1x_2)(a_3) \ \widehat{F}(y,x_1)(a_2 {}^ {x_2}a_3F(x_2,x_3))^{-1} \ \widehat{F}(y,x_1)(a_2)\\
 & \epsilon(y,x_1x_2,x_3) \  \epsilon(y,x_1,x_2x_3)^{-1}  \ \epsilon(y,x_1,x_2) \\
 =& \widehat{F}(yx_1,x_2)(a_3) \ \widehat{F}(x_1,x_2)(a_3)^{-1} \epsilon(yx_1,x_2,x_3) \ \epsilon(x_1,x_2,x_3)^{-1},
 \end{align*}
 hence we have that
 \begin{align*}
 \delta_K \mu  \cdot \delta_G \gamma^{-1} = 1
 \end{align*}
 
 Finally we calculate
 \begin{align*}
 \delta_K \gamma(k_1,k_2)&(x_1,(a_2,x_2)) \\
 =& \gamma(k_2)(x_1,(a_2,x_2)) \ \gamma(k_1k_2)(x_1,(a_2,x_2))^{-1} \ \gamma(k_1)(k_2x_1,(a_2,x_2))\\
 =& \widehat{F}(k_2,x_1)(a_2) \ \widehat{F}(k_1k_2,x_1)(a_2)^{-1} \ \widehat{F}(k_1,k_2x_2)(a_2) \\
 & \epsilon(k_2,x_1,x_2) \ \epsilon(k_1k_2,x_1,x_2)^{-1} \ \epsilon(k_1,k_2x_1,x_2)\\
 =& \widehat{F}(k_1,k_2)({}^{x_1}a_2)  \ \delta_K \epsilon(k_1,k_2,x_1,x_2) \ \epsilon(k_1,k_2,x_1x_2) \ \epsilon(k_1,k_2,x_1)^{-1}\\
 =& \widehat{F}(k_1,k_2)({}^{x_1}a_2)    \ \widehat{F}(k_1,k_2)(F(x_1,x_2)) \ \epsilon(k_1,k_2,x_1x_2)\  \epsilon(k_1,k_2,x_1)^{-1},
 \end{align*}
 and since by equation
 \eqref{varphi(widehat F) delta G beta} we have that
 \begin{align*}  (\varphi(\widehat{F}) \delta_G \beta) (k_1,k_2)(x_1,(a_2,x_2)) = &  \widehat{F}(k_1,k_2)({}^{x_1}a_2)    \ \widehat{F}(k_1,k_2)(F(x_1,x_2))\\
 &\epsilon(k_1,k_2,x_1x_2)\  \epsilon(k_1,k_2,x_1)^{-1}
 \end{align*}
 we have that
 $$\delta_K \gamma = \varphi(\widehat{F}) \delta_G \beta.$$
 Hence  $\pi^* \omega \cdot (\delta_{Tot} \mu \oplus \gamma) =\varphi(\widehat{F}) \delta_G \beta$.
   \end{proof}
   
   \subsection{Description of $\widehat{\omega}$ and $\nu$}
   Assuming the explicit descriptions of $\omega$, $\mu$ and $\gamma$ described in Theorem \ref{theorem omega mu gamma}, we see that $\tilde{\nu}= \varphi(\widehat{F}) \delta_G \beta$. Applying this explicit description of $\tilde{\nu}$ into the definition of $\nu$ given in \eqref{definition nu} and of $\widehat{\omega}$ given in \eqref{definition widehat omega}
   we obtain
   \begin{align*}
   \nu(k_1,k_2)(a):= \tilde{\nu}(k_1,k_2)(1,(a,1))= \widehat{F}(k_1,k_2)(a)
   \end{align*}
   which implies that $\nu= \widehat{F}$, and
   \begin{align*}
 \widehat{\omega}((k_1, \rho_1),  (k_2,\rho_2),(k_3 ,\rho_3)) :=& \tilde{\nu}(k_1,k_2)(1;(1,k_3)) \ \rho_1(F(k_2,k_3))\\
 =& \epsilon(k_1,k_2,k_3) \ \rho_1(F(k_2,k_3)).
   \end{align*}
   
   After applying Corollary \ref{corollary weak Morita equivalence} to the previous explicit construction of $\widehat{\omega}$ we obtain the following theorem:
   \begin{theorem} \label{theorem Morita equivalence}
   Let $K$ be a finite group acting on the finite abelian group $A$. Consider cocycles $F \in Z^2(K,A)$ and $\widehat{F} \in Z^2(K,{{\mathbb{A}}})$ such that $\widehat{F} \wedge F$ is trivial in cohomology, i.e. there exists $\epsilon \in C^3(K,\Cx)$ such that $\delta_K \epsilon = \widehat{F} \wedge F$. Define the 3-cocycles $\omega \in Z^3(A \rtimes_F K, \Cx)$ and $\widehat{\omega} \in Z^3(K \ltimes_{\widehat{F}} {{\mathbb{A}}}, \Cx)$ by the equations:
   \begin{align*}
   \omega((a_1,k_1),(a_2,k_2),(a_3,k_3)) := & \widehat{F}(k_1,k_2)(a_3) \ \epsilon(k_1,k_2,k_3)\\
    \widehat{\omega}((k_1, \rho_1),  (k_2,\rho_2),(k_3 ,\rho_3)) :=& \epsilon(k_1,k_2,k_3) \ \rho_1(F(k_2,k_3)).
   \end{align*}
   Then the tensor categories $Vect(A \rtimes_F K, \omega)$ and $Vect(K \ltimes_{\widehat{F}} {{\mathbb{A}}}, \widehat{\omega})$ are weakly Morita equivalent, and therefore their centers are braided equivalent
   $$\ZZ(Vect(A \rtimes_F K, \omega)) \simeq \ZZ(Vect(K \ltimes_{\widehat{F}} {{\mathbb{A}}}, \widehat{\omega})).$$
   \end{theorem}
   
   Note that we may have taken a different choice of $\mu$ and $\gamma$ in section \ref{subsection omega nu} thus producing different $\tilde{\nu}$ and $\widehat{\omega}$. The description of
   $\widehat{\omega}$ depends on the choice of cohomology class $[\widehat{F}] \in H^2(K,{{\mathbb{A}}})\cong E_2^{2,1}$ in the second page representing the image of
   $[\omega] $ in $E_3^{2,1}=E_\infty^{2,1}$. This choice may be changed by  
   elements in the image of the second differential $d_2: E_2^{0,2} \to E_2^{2,1}$.
   
   Changing $\omega$ by a coboundary $\omega'=\omega\delta_G \alpha$, and writing $\omega'$ explicitly as
   \begin{align}\label{def omega'}\omega'((a_1,x_1),(a_2,x_2),(a_3,x_3)) :=  \widehat{F}'(x_1,x_2)(a_3) \ \epsilon'(x_1,x_2,x_3),\end{align}
  produces a  $\widehat{\omega}'$ which becomes
  \begin{align}\label{def widehat omega'}\widehat{\omega}'((k_1, \rho_1),  (k_2,\rho_2),(k_3 ,\rho_3)) := \epsilon'(k_1,k_2,k_3) \ \rho_1(F(k_2,k_3)).\end{align}
   Applying Theorem \ref{theorem Morita equivalence} and using the equivalence of categories $Vect(A \rtimes_F K, \omega) \simeq Vect(A \rtimes_F K, \omega')$ we obtain that the tensor categories $Vect(A \rtimes_F K, \omega)$ and $Vect(K \ltimes_{\widehat{F}'} {{\mathbb{A}}}, \widehat{\omega}')$ are also weakly Morita equivalent.
   The previous argument permit us to conclude the following corollary:
   \begin{cor} \label{corollary omega'}
   Suppose that the fusion category $\CC_\MM^*=\VV(A \rtimes_F K, \omega)_{\MM(K, \mu)}^*$ is pointed. Then it is
   equivalent to the category 
   $Vect(K \ltimes_{\widehat{F}'} {{\mathbb{A}}}, \widehat{\omega}')$ where $\widehat{\omega}'$ and $\omega'$ are the cocycles defined in \eqref{def omega'} and \eqref{def widehat omega'} respectively
   and $\omega'$ is cohomologous to $\omega$.
   \end{cor}
   
   \subsection{Classification theorem} Now we are ready to state the key result in order to establish the 
 weak Morita equivalence classes of group theoretical tensor categories.

 \begin{theorem} \label{main theorem}
 Let $H$ and $\widehat{H}$ be finite groups, $\eta \in Z^3(H, \Cx)$ and $\widehat{\eta} \in Z^3(\widehat{H}, \Cx)$. Then the tensor categories $Vect(H,\eta)$ and $Vect(\widehat{H}, \widehat{\eta})$ are weakly Morita equivalent if and only if the following conditions are satified:
 \begin{itemize}
 \item There exist isomorphisms of groups 
 $$\phi : G= A \rtimes_F K \stackrel{\cong}{\to} H \ \ \ \ 
 \widehat{\phi} : \widehat{G}= K \ltimes_{\widehat{F}} {{\mathbb{A}}} \stackrel{\cong}{\to} \widehat{H}$$
 for some finite group $K$ acting on the abelian group $A$,
 with $F \in Z^2(K, A)$ and $\widehat{F} \in Z^2(K, {{\mathbb{A}}})$ where ${{\mathbb{A}}} := \Hom(A, \Cx)$. 
 \item There exist $\epsilon : K^3 \to \Cx$ such that $\widehat{F} \wedge F =\delta_K \epsilon$.
 \item The cohomology classes satisfy the equations $[ \phi^* \eta ]=[\omega]$ and 
 $[\widehat{\phi}^*\widehat{\eta}]=[\widehat{\omega}]$ with 
 \begin{align*}
   \omega((a_1,k_1),(a_2,k_2),(a_3,k_3)) := & \widehat{F}(k_1,k_2)(a_3) \ \epsilon(k_1,k_2,k_3)\\
    \widehat{\omega}((k_1, \rho_1),  (k_2,\rho_2),(k_3 ,\rho_3)) :=& \epsilon(k_1,k_2,k_3) \ \rho_1(F(k_2,k_3)).
   \end{align*}
 
 \end{itemize}
 \end{theorem}
 \begin{proof}
 Suppose that $Vect(H,\eta)$ and $Vect(\widehat{H}, \widehat{\eta})$ are weaky Morita equivalent. Then
 $Vect(\widehat{H}, \widehat{\eta})$ is equivalent 
 to the dual category $\VV(H,\eta)_{\MM(A \backslash H, \mu)}^*$ with
  $K:= A \backslash H$, $\phi : G= A \rtimes_F K \stackrel{\cong}{\to} H$ and  $\MM(A \backslash H, \mu)$
   some module category of $\VV(H,\eta)$. By Corollary \ref{corollary omega'} the tensor category 
   $Vect(\widehat{H}, \widehat{\eta})$  is furthermore equivalent
 to $Vect(K \ltimes_{\widehat{F}'} {{\mathbb{A}}}, \widehat{\omega}')$ where $\omega'$ and
  $\widehat{\omega}'$ are the cocycles defined in equations \eqref{def omega'} and \eqref{def widehat omega'} respectively, and such that $\omega'$ is cohomologous to $\phi^*\eta$.  In particular we have that
$ \widehat{\phi} : \widehat{G}= K \ltimes_{\widehat{F}} {{\mathbb{A}}} \stackrel{\cong}{\to} \widehat{H}$ and that
 $\widehat{\phi}^*\widehat{\eta}$ is cohomologous to $\widehat{\omega}'$.

 The converse is the statement of Theorem \ref{theorem Morita equivalence}.
    \end{proof}
    In the case that both $\omega$ and $\widehat{\omega}$ were cohomologically trivial, we conclude that
    $Vect(A \rtimes_F K, 1)$ and $Vect(K \ltimes_{\widehat{F}} {{\mathbb{A}}}, 1)$
    are weakly Morita equivalent if and only if the cohomology class $[\widehat{F}] \in H^2(K,{{\mathbb{A}}})$ lies in the image
    of the second differential of the spectral sequence
    $d_2: H^2(A, \Cx)^K \to H^2(K,{{\mathbb{A}}}).$
    This result was originally proved in \cite[Cor. 6.2]{Davydov}.
    
    \section{Examples} \label{section examples}
    \subsection{Pointed fusion categories of global dimension 4}
    We can now calculate the weakly Morita equivalence classes of pointed fusion categories of global dimension 4. 
    
    For $G = \integer/4$ we  have that
    $H^*(\integer/4,\integer) \cong \integer[u]/4u$ with $|u|=2$ and that the non trivial automorphism of $\integer/4$ maps $u$ to $-u$; therefore
    $H^4(\integer/4,\integer)/Aut(\integer/4) =\langle u^2 \rangle = \integer/4$.
    
     For $G= (\integer/2)^2$ we have that
    $$H^4((\integer/2)^2, \integer) \cong ker(Sq^1: H^4((\integer/2)^2, \IF_2) \to H^5((\integer/2)^2, \IF_2)) =\langle x^4,x^2y^2,y^4 \rangle$$
    where $H^*((\integer/2)^2, \IF_2)=\IF_2[x,y]$ and $Sq^1$ is the Steenrod operation, and up to automorphisms of $(\integer/2)^2$ we get
    $$
    H^4((\integer/2)^2, \integer)/Aut((\integer/2)^2) = \left\{ 
    \begin{array}{l}
    0\\
    (x^4)= \{x^4,y^4,x^4+y^4 \} \\
    (x^2y^2) =\{ x^2y^2,x^2y^2+x^4,x^2y^2+y^4 \}\\
    (x^4+x^2y^2+y^4) = \{x^4+x^2y^2+y^4\}.
    \end{array}
    \right.
    $$
    
    Since we have a clear description for a base of $H^4((\integer/2)^2, \integer)$,
    we will abuse the notation and denote with the symbols of 
    $H^4((\integer/2)^2, \integer)$ the elements of $H^3((\integer/2)^2, \Cx)$. With this 
    clarification
    the relevant terms of the second page of the LHS spectral sequence of the extension $1 \to \integer/2 \to \integer/4 \to \integer/2 \to 1$
    become
    
    \begin{tikzpicture}
\matrix [matrix of math nodes,row sep=6mm]
{
 3 &  [5mm]  |(a)|  \integer/2= \langle y^4 \rangle & [5mm]   & [5mm]  & [5mm] & [5mm] & [5mm] \\\
2 & |(b)| 0 & |(c)|  0  &  & & & \\
1&  \integer/2 & |(d)|  \integer/2= \langle yx\rangle & |(e)| \integer/2= \langle yx^2\rangle &  & & \\
0& \Cx &  \integer/2 & |(f)| 0 & |(g)|  \integer/2=\langle x^4\rangle & 0 \\
& 0 & 1 & 2 & 3& 4&\\
};

\tikzstyle{every node}=[midway,auto,font=\scriptsize]
\draw[thick] (-5,-1.7) -- (-5,2.8) ;
\draw[thick] (-5,-1.7) -- (6.0,-1.7) ;
\draw[-stealth] (d) -- node {$\cong$} (g);
\end{tikzpicture}

\noindent where the second differential is defined by the assignment $d_2(yx^k)= Sq^1(x^{k+2})$ with the class $x^2$ classifying the extension.
We conclude that the only weak Morita equivalence that appears, which does not come from an automorphism of a group, is
$$Vect(\integer/4, 0) \simeq Vect((\integer/2)^2,x^2y^2).$$
    
    Therefore we see that there are exactly seven weak Morita equivalence classes of pointed fusion categories of global dimension 4, namely
    the three for $\integer/4$: $$ Vect(\integer/4,u^2),
    Vect(\integer/4,2u^2),
    Vect(\integer/4,3u^2);$$ the three for $(\integer/2)^2$: $$Vect((\integer/2)^2,0),Vect((\integer/2)^2,x^4),Vect((\integer/2)^2, x^4+y^4+x^2y^2)$$
    and the one that we have just constructed
   $$ Vect(\integer/4,0)\simeq_M Vect((\integer/2)^2,x^2y^2).    $$
        
    \subsection{Non trivial action of $\integer/2$ on $\integer/4$} \label{subsection nontrivial action}
    Consider the non trivial action of $\integer/2$ on $\integer/4$ and the abelian extension $1 \to \integer/4 \to G \to \integer/2 \to 1$. The group
    $G$ is either the Dihedral group $D_8$ in the case that the extension is a split extension or the quaternion group $Q_8$ in the case that the extension is a non-split extension.
    
    In the case of $D_8$ the relevant elements of the second page of the LHS spectral sequence
    associated to the extension are:
    
    \begin{tikzpicture}
\matrix [matrix of math nodes,row sep=6mm]
{
 3 &  [5mm]  |(a)|  \integer/4= \langle a\rangle & [5mm]   & [5mm]  & [5mm] & [5mm] & [5mm] \\\
2 & |(b)| 0 & |(c)|  0  &  & & & \\
1&  \integer/2 & |(d)|  \integer/2= \langle e\rangle & |(e)| \integer/2= \langle b\rangle &  & & \\
0& \Cx &  \integer/2 & |(f)| 0 & |(g)|  \integer/2=\langle c\rangle & 0 \\
& 0 & 1 & 2 & 3& 4&\\
};
\tikzstyle{every node}=[midway,auto,font=\scriptsize]
\draw[thick] (-4.5,-1.7) -- (-4.5,2.8) ;
\draw[thick] (-4.5,-1.7) -- (5.0,-1.7) ;
\end{tikzpicture}

 \noindent   and they all survive to the page at infinity.  Since $H^3(D_8, \Cx)= \integer/4 \oplus \integer/2 \oplus \integer/2$ we may say that
    $H^3(D_8, \Cx)\cong \langle a \rangle \oplus \langle b \rangle  \oplus \langle c \rangle$, and since $D_8 \cong \integer/4 \rtimes \integer/2$
   we have that $F=0$. The element $b \in H^2(\integer/2, {\integer/4})$ defines the non trivial extension $Q_8 \cong \integer/2 \ltimes_b {\integer/4}$.
   
   The second page of the LHS spectral sequence of the extension
   $Q_8 \cong \integer/2 \ltimes_b {\integer/4}$ becomes:
   
    \begin{tikzpicture}
\matrix [matrix of math nodes,row sep=6mm]
{
 3 &  [5mm]  |(a)|  \integer/4= \langle \alpha \rangle & [5mm]   & [5mm]  & [5mm] & [5mm] & [5mm] \\\
2 & |(b)| 0 & |(c)|  0  &  & & & \\
1&  \integer/2 & |(d)|  \integer/2= \langle e\rangle & |(e)| \integer/2= \langle 4\alpha\rangle &  & & \\
0& \Cx &  \integer/2 & |(f)| 0 & |(g)|  \integer/2=\langle c\rangle & 0 \\
& 0 & 1 & 2 & 3& 4&\\
};
\tikzstyle{every node}=[midway,auto,font=\scriptsize]
\draw[thick] (-4.5,-1.7) -- (-4.5,2.8) ;
\draw[thick] (-4.5,-1.7) -- (5.0,-1.7) ;
\draw[-stealth] (d) -- node {$\cong$} (g);
\end{tikzpicture}

\noindent   where $d_2: E_2^{1,1} \stackrel{\cong}{\to} E_2^{3,0}$ is an isomorphism and $H^3(Q_8,\Cx)=\langle \alpha \rangle= \integer/8$.
   
   Therefore for these extensions we only have the weak Morita equivalences:
   \begin{align*}
   Vect(D_8, b) \simeq_M Vect(Q_8,0) \simeq_M  Vect(D_8,b \oplus c) \end{align*} 
   where the equivalence of the right is obtained from the fact that $c$ does not survive the spectral sequence for the group $Q_8$, and the self Morita equivalence
$$   Vect(Q_8,4\alpha) \simeq_M Vect(Q_8,4\alpha).$$
   \subsection{Extension of $\integer/2 \times \integer/2$ by $\integer/2$}
   Consider the non-abelian extensions of the form $1 \to \integer/2 \to G \to \integer/2 \times \integer/2 \to 1$, namely $D_8$ and $Q_8$.
   
   The second page of the LHS spectral sequence for these extensions becomes:
   
    \begin{tikzpicture}
\matrix [matrix of math nodes,row sep=6mm]
{
 3 &  [5mm]  |(a)|  \integer/2 & [5mm]   & [5mm]  & [5mm] & [5mm] & [5mm] \\\
2 & |(b)| 0 & |(c)|  0  &  & & & \\
1&  \integer/2 & |(d)|  (\integer/2)^2 & |(e)| (\integer/2)^3 &  & & \\
0& \Cx &  (\integer/2)^2 & |(f)| \integer/2 & |(g)|  (\integer/2)^3 & (\integer/2)^2 \\
& 0 & 1 & 2 & 3& 4&\\
};
\tikzstyle{every node}=[midway,auto,font=\scriptsize]
\draw[thick] (-4.0,-1.7) -- (-4.0,2.8) ;
\draw[thick] (-4.0,-1.7) -- (4.0,-1.7) ;
\end{tikzpicture}
   
  \noindent and we need only to concentrate in the differentials $d_2:E_2^{p,1} \to E_2^{p+2,0}$ between the first two rows since we know that $E_2^{0,3} = \integer/2$ survives the spectral sequence in all the groups.
   
   First we will determine the differential $\overline{d}_2^G$ in the LHS spectral sequence for coefficients in the field of two elements $\IF_2$. In this
   case $$E_2 \cong H^*(\integer/2 \times \integer/2, \IF_2) \otimes_{\IF_2} H^*(\integer/2, \IF_2) \cong \IF_2[x,y,e]$$
   and $\overline{d}_2^Ge \in H^2(\integer/2 \times \integer/2, \IF_2)$ represents
   the class that defines the extension $G$. It is known  
   that the class $x^2+xy +y^2$ defines $Q_8$ \cite[Lemma 2.10]{AdemMilgram}, the classes $x^2+xy, xy+y^2,xy$ define $D_8$ \cite[pp. 130]{AdemMilgram}, and the classes $x^2,y^2,x^2+y^2$ define
   $\integer/2 \times \integer/4$.
   
   Second we use the fact that for the group $(\integer/2)^2$  we have the isomorphism
   $$H^j((\integer/2)^2, \integer) \cong ker(Sq^1 : H^j((\integer/2)^2, \integer/2) \to H^{j+1}((\integer/2)^2, \integer/2))  $$
   where $Sq^1$ is the first Steenrod square.  This implies that the map
   canonical map
   $$H^j((\integer/2)^2, \integer/2)) \to H^{j}((\integer/2)^2, \Cx)$$
   can be seen as the map
   \begin{align*}
   H^j((\integer/2)^2, \integer/2)) \stackrel{Sq^1}{\longrightarrow} &
   ker\left(Sq^1 : H^{j+1}((\integer/2)^2, \integer/2) \to H^{j+2}((\integer/2)^2, \integer/2)\right)\\
  & \cong H^{j+1}((\integer/2)^2, \integer) \cong H^{j}((\integer/2)^2, \complex^*).
   \end{align*} 
   
   Therefore the second differential
   $$d_2^G : H^{p-2}((\integer/2)^2, \integer/2) \to H^{p}((\integer/2)^2, \Cx)$$
   is isomorphic to the composite map
   \begin{align*}
   d_2^G : H^{p-2}((\integer/2)^2, \integer/2)  \to  & ker\left(Sq^1 : H^{p+1}((\integer/2)^2, \integer/2) \to H^{p+2}((\integer/2)^2, \integer/2)\right)\\
   & \cong H^{p+1}((\integer/2)^2, \integer) \cong H^{p}((\integer/2)^2, \complex^*)\\
   z  \mapsto & Sq^1(z \cup \overline{d}_2^Ge).
   \end{align*}
   
   Without loss of generality we may choose $\overline{d}_2^Ge=xy+x^2$ for calculating the LHS spectral sequence for $D_8$. Applying the differential $d_2^G$ to the elements $1,x,y,x^2,xy,y^2$ we obtain that the surviving terms in the infinite page of the LHS spectral sequence for $D_8$ become:
   
    \begin{tikzpicture}
\matrix [matrix of math nodes,row sep=6mm]
{
 3 &  [1mm]  |(a)|  \integer/2 & [1mm]   & [1mm]  & [1mm] & [1mm] & [1mm] \\\
2 & |(b)| 0 & |(c)|  0  &  & & & \\
1&  0 & |(d)|  \integer/2=\langle e(y) \rangle  & |(e)| \integer/2=\langle e(xy+x^2) \rangle  &  & & \\
0& \Cx &  (\integer/2)^2=\langle x^2,y^2 \rangle & |(f)| 0  & |(g)|  (\integer/2)^2=\frac{\langle x^4,x^2y^2,y^4 \rangle}{\langle x^2y^2+x^4 \rangle} & 0 \\
& 0 & 1 & 2 & 3& 4&\\
};

\tikzstyle{every node}=[midway,auto,font=\scriptsize]
\draw[thick] (-5.3,-1.7) -- (-5.3,2.8) ;
\draw[thick] (-5.3,-1.7) -- (5.9,-1.7) ;
\end{tikzpicture}
   
\noindent  Here we are abusing the notation and we are using the explicit
base of $H^4((\integer/2)^2, \integer)$ to denote the elements in 
$H^3((\integer/2)^2, \Cx)$.
Since $E_3^{2,1}= \langle e(xy+x^2) \rangle$ we have that the weak Morita
   equivalences that we obtain in the extension are 
   $$Vect(D_8,0) \simeq_M Vect((\integer/2)^3, Sq^1(e(xy+x^2)))$$
    $$Vect(D_8,x^4) \simeq_M Vect((\integer/2)^3, Sq^1(e(xy+x^2))+x^4)$$
     $$Vect(D_8,y^4) \simeq_M Vect((\integer/2)^3, Sq^1(e(xy+x^2))+y^4)$$
   and the self equivalence
   $$Vect(D_8,e(xy+x^2) \simeq Vect(D_8,e(xy+x^2).$$

   The surviving terms for $Q_8$ with $\overline{d}_2^Ge=x^2+xy+y^2$ are:
   
   \begin{tikzpicture}
\matrix [matrix of math nodes,row sep=6mm]
{
 3 &  [3mm]  |(a)|  \integer/2 & [3mm]   & [3mm]  & [3mm] & [3mm] & [3mm] \\\
2 & |(b)| 0 & |(c)|  0  &  & & & \\
1&  0 & |(d)|  0  & |(e)| \integer/2=\langle e(x^2 +xy+ y^2) \rangle &  & & \\
0& \Cx &  (\integer/2)^2=\langle x^2,y^2 \rangle & |(f)| 0  & |(g)|  \integer/2=\langle x^2y^2\rangle & 0 \\
& 0 & 1 & 2 & 3& 4&\\
};
\tikzstyle{every node}=[midway,auto,font=\scriptsize]
\draw[thick] (-5.6,-1.7) -- (-5.6,2.8) ;
\draw[thick] (-5.6,-1.7) -- (6.1,-1.7) ;
\end{tikzpicture}
   
   \noindent with $E_\infty^{0,3} = \integer/2=\langle \alpha \rangle$,
   $\langle x^2 +xy+ y^2 \rangle= \langle 2\alpha \rangle$ and 
   $\langle x^2y^2 \rangle= \langle 4\alpha \rangle$ where 
   $\alpha$ is a generator $\langle \alpha \rangle = H^3(Q_8,\Cx)$ that was defined in section \S\ref{subsection nontrivial action}.
   
 Hence the only Morita equivalences that we obtain are 
 $$Vect(Q_8,0) \simeq Vect((\integer/2)^3, Sq^1(e(x^2 +xy+ y^2)))$$
 $$Vect(Q_8,4\alpha) \simeq Vect((\integer/2)^3, Sq^1(e(x^2 +xy+ y^2))+x^2y^2)$$
 and the self Morita equivalences
   $Vect(Q_8,2\alpha) \simeq_M Vect(Q_8,2\alpha) $ and $Vect(Q_8,6\alpha) \simeq_M Vect(Q_8,6\alpha) $.
   
   Bundling up the previous results for the group $Q_8$ we obtain the following result:
   
   \begin{proposition}
Let us suppose that $Vect(Q_8, k\alpha)$ is weakly Morita equivalent to $Vect(G,\eta)$. Then 
\begin{itemize}
\item For $k$ odd or  $k=2,6$, $G$ must be isomorphic to $Q_8$ and $\eta$ must correspond to $j\alpha$ with $j$ odd or $j=2,6$.
\item For $k=4$, $G$ must be isomorphic to $Q_8$ or $(\integer/2)^3$.
\item For $k=0$, $G$ must be isomorphic to $Q_8$, $D_8$ or $(\integer/2)^3$.
\end{itemize}
   \end{proposition}
   \begin{proof}
   First note that the action of $Aut(Q_8)$ on $H^3(Q_8,\Cx)$ is trivial.
   Second note that the only normal subgroups of $Q_8$ are its center and the cyclic ones generated by roots of unity and that they
   all fit into the central extension $1 \to \integer/2 \to Q_8 \to (\integer/2)^2 \to 1$ or the non-split extension $1 \to \integer/4 \to Q_8 \to \integer/2 \to 1$ that we have studied before. Since any weak Morita equivalence between pointed fusion categories comes from
   a normal and abelian subgroup of $Q_8$, the classification that we have done before exhausts all possibilities. For $k$ odd
   we know that $k\alpha$ survives to the restriction to the center and to the cyclic subgroups isomorphic to $\integer/4$ and therefore
   $G$ can only be $Q_8$. The classes $2\alpha$ and
   $6\alpha$ trivialize on the center of $Q_8$ but these classes define extensions of $(\integer/2)^2$ by $\integer/2$ which are
   isomorphic to $Q_8$ and define cohomology classes which are precisely $2\alpha$ and $6\alpha$. The class $4\alpha$
   trivializes in all normal and abelian subgroups; in the case of the subgroup $\integer/4$ the only group that may appear is $Q_8$,
    and in the case of the center we may obtain the weak Morita equivalence
    $$Vect(Q_8,4\alpha) \simeq Vect((\integer/2)^3, Sq^1(e(x^2 +xy+ y^2))+x^2y^2).$$
   Finally, the trivial class produces only the group $D_8$ in the case of the subgroup $\integer/4$ and $(\integer/2)^3$ in the case of the center;
   some weak Morita equivalences are
   $$Vect(Q_8,0) \simeq Vect((\integer/2)^3, Sq^1(e(x^2 +xy+ y^2))) \simeq_M  Vect(D_8,b).$$
   
      \end{proof}

\bibliographystyle{abbrv} 

\begin{thebibliography}{10}

\bibitem{AdemMilgram}
A.~Adem and R.~J. Milgram.
\newblock {\em Cohomology of finite groups}, volume 309 of {\em Grundlehren der
  Mathematischen Wissenschaften [Fundamental Principles of Mathematical
  Sciences]}.
\newblock Springer-Verlag, Berlin, 1994.

\bibitem{Bakalov-Kirillov}
B.~Bakalov and A.~Kirillov, Jr.
\newblock {\em Lectures on tensor categories and modular functors}, volume~21
  of {\em University Lecture Series}.
\newblock American Mathematical Society, Providence, RI, 2001.

\bibitem{Davydov}
A.~A. Davydov.
\newblock Finite groups with the same character tables, {D}rinfel\cprime d
  algebras and {G}alois algebras.
\newblock In {\em Algebra ({M}oscow, 1998)}, pages 99--111. de Gruyter, Berlin,
  2000.

\bibitem{Dijkgraaf}
R.~Dijkgraaf, V.~Pasquier, and P.~Roche.
\newblock Quasi {H}opf algebras, group cohomology and orbifold models.
\newblock {\em Nuclear Phys. B Proc. Suppl.}, 18B:60--72 (1991), 1990.
\newblock Recent advances in field theory (Annecy-le-Vieux, 1990).

\bibitem{ENO}
P.~Etingof, D.~Nikshych, and V.~Ostrik.
\newblock On fusion categories.
\newblock {\em Ann. of Math. (2)}, 162(2):581--642, 2005.

\bibitem{W-GT-fusion-cat}
P.~Etingof, D.~Nikshych, and V.~Ostrik.
\newblock Weakly group-theoretical and solvable fusion categories.
\newblock {\em Adv. Math.}, 226(1):176--205, 2011.

\bibitem{Evens}
L.~Evens.
\newblock {\em The cohomology of groups}.
\newblock Oxford Mathematical Monographs. The Clarendon Press, Oxford
  University Press, New York, 1991.
\newblock Oxford Science Publications.

\bibitem{Movshev}
M.~V. Movshev.
\newblock Twisting in group algebras of finite groups.
\newblock {\em Funktsional. Anal. i Prilozhen.}, 27(4):17--23, 95, 1993.

\bibitem{MugerI}
M.~M{\"u}ger.
\newblock From subfactors to categories and topology. {I}. {F}robenius algebras
  in and {M}orita equivalence of tensor categories.
\newblock {\em J. Pure Appl. Algebra}, 180(1-2):81--157, 2003.

\bibitem{Naidu}
D.~Naidu.
\newblock Categorical {M}orita equivalence for group-theoretical categories.
\newblock {\em Comm. Algebra}, 35(11):3544--3565, 2007.

\bibitem{Ost-2}
V.~Ostrik.
\newblock Module categories over the {D}rinfeld double of a finite group.
\newblock {\em Int. Math. Res. Not.}, (27):1507--1520, 2003.

\bibitem{Ost}
V.~Ostrik.
\newblock Module categories, weak {H}opf algebras and modular invariants.
\newblock {\em Transform. Groups}, 8(2):177--206, 2003.

\end{thebibliography}

\def\cprime{$'$} \def\cprime{$'$}

\end{document}